\documentclass[12pt,reqno]{amsart} 
\usepackage{amssymb,amscd,url}

\begin{document}

\allowdisplaybreaks

\hyphenation{ca-non-i-cal}



\newtheorem{theorem}{Theorem}
\newtheorem{lemma}[theorem]{Lemma}
\newtheorem{conjecture}[theorem]{Conjecture}
\newtheorem{question}[theorem]{Question}
\newtheorem{proposition}[theorem]{Proposition}
\newtheorem{corollary}[theorem]{Corollary}
\newtheorem*{claim}{Claim}

\theoremstyle{definition}
\newtheorem*{definition}{Definition}
\newtheorem{example}[theorem]{Example}
\newtheorem{remark}[theorem]{Remark}

\theoremstyle{remark}
\newtheorem*{acknowledgement}{Acknowledgements}



\newenvironment{notation}[0]{%
  \begin{list}%
    {}%
    {\setlength{\itemindent}{0pt}
     \setlength{\labelwidth}{4\parindent}
     \setlength{\labelsep}{\parindent}
     \setlength{\leftmargin}{5\parindent}
     \setlength{\itemsep}{0pt}
     }%
   }%
  {\end{list}}

\newenvironment{parts}[0]{%
  \begin{list}{}%
    {\setlength{\itemindent}{0pt}
     \setlength{\labelwidth}{1.5\parindent}
     \setlength{\labelsep}{.5\parindent}
     \setlength{\leftmargin}{2\parindent}
     \setlength{\itemsep}{0pt}
     }%
   }%
  {\end{list}}
\newcommand{\Part}[1]{\item[\upshape#1]}

\renewcommand{\a}{\alpha}
\renewcommand{\b}{\beta}
\newcommand{\g}{\gamma}
\renewcommand{\d}{\delta}
\newcommand{\e}{\epsilon}
\newcommand{\f}{\varphi}
\newcommand{\bfphi}{{\boldsymbol{\f}}}
\renewcommand{\l}{\lambda}
\renewcommand{\k}{\kappa}
\newcommand{\lhat}{\hat\lambda}
\newcommand{\m}{\mu}
\newcommand{\bfmu}{{\boldsymbol{\mu}}}
\renewcommand{\o}{\omega}
\renewcommand{\r}{\rho}
\newcommand{\rbar}{{\bar\rho}}
\newcommand{\s}{\sigma}
\newcommand{\sbar}{{\bar\sigma}}
\renewcommand{\t}{\tau}
\newcommand{\z}{\zeta}

\newcommand{\D}{\Delta}
\newcommand{\G}{\Gamma}
\newcommand{\F}{\Phi}

\newcommand{\ga}{{\mathfrak{a}}}
\newcommand{\gb}{{\mathfrak{b}}}
\newcommand{\gn}{{\mathfrak{n}}}
\newcommand{\gp}{{\mathfrak{p}}}
\newcommand{\gP}{{\mathfrak{P}}}
\newcommand{\gq}{{\mathfrak{q}}}

\newcommand{\Abar}{{\bar A}}
\newcommand{\Ebar}{{\bar E}}
\newcommand{\Kbar}{{\bar K}}
\newcommand{\Pbar}{{\bar P}}
\newcommand{\Sbar}{{\bar S}}
\newcommand{\Tbar}{{\bar T}}
\newcommand{\ybar}{{\bar y}}
\newcommand{\phibar}{{\bar\f}}

\newcommand{\Acal}{{\mathcal A}}
\newcommand{\Bcal}{{\mathcal B}}
\newcommand{\Ccal}{{\mathcal C}}
\newcommand{\Dcal}{{\mathcal D}}
\newcommand{\Ecal}{{\mathcal E}}
\newcommand{\Fcal}{{\mathcal F}}
\newcommand{\Gcal}{{\mathcal G}}
\newcommand{\Hcal}{{\mathcal H}}
\newcommand{\Ical}{{\mathcal I}}
\newcommand{\Jcal}{{\mathcal J}}
\newcommand{\Kcal}{{\mathcal K}}
\newcommand{\Lcal}{{\mathcal L}}
\newcommand{\Mcal}{{\mathcal M}}
\newcommand{\Ncal}{{\mathcal N}}
\newcommand{\Ocal}{{\mathcal O}}
\newcommand{\Pcal}{{\mathcal P}}
\newcommand{\Qcal}{{\mathcal Q}}
\newcommand{\Rcal}{{\mathcal R}}
\newcommand{\Scal}{{\mathcal S}}
\newcommand{\Tcal}{{\mathcal T}}
\newcommand{\Ucal}{{\mathcal U}}
\newcommand{\Vcal}{{\mathcal V}}
\newcommand{\Wcal}{{\mathcal W}}
\newcommand{\Xcal}{{\mathcal X}}
\newcommand{\Ycal}{{\mathcal Y}}
\newcommand{\Zcal}{{\mathcal Z}}

\renewcommand{\AA}{\mathbb{A}}
\newcommand{\BB}{\mathbb{B}}
\newcommand{\CC}{\mathbb{C}}
\newcommand{\FF}{\mathbb{F}}
\newcommand{\GG}{\mathbb{G}}
\newcommand{\NN}{\mathbb{N}}
\newcommand{\PP}{\mathbb{P}}
\newcommand{\QQ}{\mathbb{Q}}
\newcommand{\RR}{\mathbb{R}}
\newcommand{\ZZ}{\mathbb{Z}}

\newcommand{\bfa}{{\boldsymbol a}}
\newcommand{\bfb}{{\boldsymbol b}}
\newcommand{\bfc}{{\boldsymbol c}}
\newcommand{\bfe}{{\boldsymbol e}}
\newcommand{\bff}{{\boldsymbol f}}
\newcommand{\bfg}{{\boldsymbol g}}
\newcommand{\bfj}{{\boldsymbol j}}
\newcommand{\bfp}{{\boldsymbol p}}
\newcommand{\bfr}{{\boldsymbol r}}
\newcommand{\bfs}{{\boldsymbol s}}
\newcommand{\bft}{{\boldsymbol t}}
\newcommand{\bfu}{{\boldsymbol u}}
\newcommand{\bfv}{{\boldsymbol v}}
\newcommand{\bfw}{{\boldsymbol w}}
\newcommand{\bfx}{{\boldsymbol x}}
\newcommand{\bfy}{{\boldsymbol y}}
\newcommand{\bfz}{{\boldsymbol z}}
\newcommand{\bfA}{{\boldsymbol A}}
\newcommand{\bfF}{{\boldsymbol F}}
\newcommand{\bfB}{{\boldsymbol B}}
\newcommand{\bfD}{{\boldsymbol D}}
\newcommand{\bfG}{{\boldsymbol G}}
\newcommand{\bfI}{{\boldsymbol I}}
\newcommand{\bfM}{{\boldsymbol M}}
\newcommand{\bfT}{{\boldsymbol T}}
\newcommand{\bfzero}{{\boldsymbol{0}}}

\newcommand{\Aut}{\operatorname{Aut}}
\newcommand{\CM}{\operatorname{CM}}   
\newcommand{\Disc}{\operatorname{Disc}}
\newcommand{\Div}{\operatorname{Div}}
\newcommand{\Ell}{\operatorname{Ell}}   
\newcommand{\End}{\operatorname{End}}
\newcommand{\Fix}{\operatorname{Fix}}
\newcommand{\Fbar}{{\bar{F}}}
\newcommand{\Gal}{\operatorname{Gal}}
\newcommand{\GL}{\operatorname{GL}}
\newcommand{\Hom}{\operatorname{Hom}}
\newcommand{\Index}{\operatorname{Index}}
\newcommand{\Image}{\operatorname{Image}}
\newcommand{\liftable}{{\textup{liftable}}}
\newcommand{\hhat}{{\hat h}}
\newcommand{\Ker}{{\operatorname{ker}}}
\newcommand{\Lift}{\operatorname{Lift}}
\newcommand{\MOD}[1]{~(\textup{mod}~#1)}
\newcommand{\Norm}{{\operatorname{\mathsf{N}}}}
\newcommand{\notdivide}{\nmid}
\newcommand{\nondeg}{{\textup{nondeg}}}
\newcommand{\mnondeg}{{\textup{$m$-nondeg}}}
\newcommand{\normalsubgroup}{\triangleleft}
\newcommand{\odd}{{\operatorname{odd}}}
\newcommand{\onto}{\twoheadrightarrow}
\newcommand{\ord}{\operatorname{ord}}
\newcommand{\Per}{\operatorname{Per}}
\newcommand{\PrePer}{\operatorname{PrePer}}
\newcommand{\PGL}{\operatorname{PGL}}
\newcommand{\Pic}{\operatorname{Pic}}
\newcommand{\PolCoeff}[1]{{\boldsymbol{[\![}}#1{\boldsymbol{]\!]}}}
\newcommand{\Prob}{\operatorname{Prob}}
\newcommand{\proj}{\operatorname{proj}}
\newcommand{\Qbar}{{\bar{\QQ}}}
\newcommand{\rank}{\operatorname{rank}}
\newcommand{\Rat}{\operatorname{Rat}}
\newcommand{\Ratbar}{\overline{\operatorname{Rat}}}
\newcommand{\Resultant}{\operatorname{Res}}
\renewcommand{\setminus}{\smallsetminus}
\newcommand{\Span}{\operatorname{Span}}
\newcommand{\tors}{{\textup{tors}}}
\newcommand{\Trace}{\operatorname{Trace}}
\newcommand{\Uhat}{{\hat U}}
\newcommand{\UHP}{{\mathfrak{h}}}    
\newcommand{\wt}{\operatorname{wt}}
\newcommand{\<}{\langle\!\langle}
\renewcommand{\>}{\rangle\!\rangle}

\newcommand{\longhookrightarrow}{\lhook\joinrel\longrightarrow}
\newcommand{\longonto}{\relbar\joinrel\twoheadrightarrow}

\newcommand{\Spec}{\operatorname{Spec}}
\newcommand{\gm}{{\mathfrak{m}}}
\newcommand{\gA}{{\mathfrak{A}}}
\newcommand{\gF}{{\mathfrak{F}}}
\newcommand{\gG}{{\mathfrak{G}}}
\newcommand{\gC}{{\mathfrak{C}}}
\newcommand{\gR}{{\mathfrak{R}}}
\newcommand{\gS}{{\mathfrak{S}}}
\newcommand{\zero}{{\operatorname{div}}}
\renewcommand{\div}{{\operatorname{div}}}
\newcommand{\Aff}{{\mathbb{A}}}

\title[Landen transforms as families of maps]
{Landen transforms as families of (commuting)
rational self-maps of projective space}
\date{24 August, 2013}

\author[Joyce, Kawaguchi, Silverman]
  {Michael Joyce, Shu Kawaguchi, and Joseph H. Silverman}
\email{\url{mjoyce3@tulane.edu, kawaguch@math.kyoto-u.ac.jp,
    jhs@math.brown.edu}}
\address{Department of Mathematics, Tulane University, New Orleans, 
         LA 70118 USA}
\address{Department of Mathematics, Faculty of Science, Kyoto University, 
         Kyoto, 606-8502, Japan}
\address{Mathematics Department, Box 1917
         Brown University, Providence, RI 02912 USA}

\thanks{Kawaguchi's research supported by KAKENHI 24740015.
Silverman's research supported by NSA H98230-04-1-0064,
NSF DMS-0854755, and Simons Collaboration Grant \#241309}

\begin{abstract}
The classical $(m,k)$-Landen transform~$\gF_{m,k}$ is a self-map of
the field of rational function~$\CC(z)$ obtained by forming a weighted
average of a rational function over twists by $m$'th roots of
unity. Identifying the set of rational maps of degree~$d$ with an
affine open subset of~$\PP^{2d+1}$, we prove that~$\gF_{m,0}$ induces
a dominant rational self-map~$\gR_{d,m,0}$ of~$\PP^{2d+1}$ of
algebraic degree~$m$, and for~$1\le k<m$, the transform~$\gF_{m,k}$
induces a dominant rational self-map~$\gR_{d,m,k}$ of algebraic
degree~$m$ of a certain hyperplane in~$\PP^{2d+1}$. We show in all
cases that~$\gR_{d,m,k}$ extends nicely  to~$\PP^{2d+1}_\ZZ$,
and that~$\{\gR_{d,m,0} : m\ge0\}$ is a commuting family of maps.
\end{abstract}


\maketitle

\tableofcontents

\section{Introduction}
The Landen transform, also known as Gauss' arithmetic-geometric mean,
is a self-map of the space of rational functions in one variable.  The
purpose of this note is to study the generalized Landen transform from
the viewpoint of arithmetic geometry and arithmetic dynamics.  We
defer until Section~\ref{section:history} a discussion of the history
and historical applications of the Landen transform, and devote
this introduction to describing our main results. 

The following proposition characterizes the generalized Landen
transform.

\begin{proposition}
\label{proposition:Gphicharac}
Let $m\ge1$ and $0\le k<m$ be integers, let~$K$ be a field in which
$m\ne0$, and let~$\z_m$ be a primitive $m$'th root of unity in an
extension field of~$K$. Then for each rational function~$\f(z)\in
K(z)$ there is a unique rational function $\gF_{m,k}(\f)(z)\in K(z)$,
called the \emph{$(m,k)$-Landen transform of~$\f$}, characterized by
the formula
\begin{equation}
  \label{eqn:Fmkfwmdef}
  \gF_{m,k}(\f)(w^m) =  \frac{1}{mw^k}
     \sum_{t=0}^{m-1} \z_m^{-kt}\f(\z_m^t w).
\end{equation}
\end{proposition}

\noindent
As a warm-up for our main result, we give the elementary proof of
Proposition~\ref{proposition:Gphicharac} in
Section~\ref{section:history}; see
Proposition~\ref{proposition:FphiinKwm}.

We denote the space of rational functions of degree~$d$ by~$\Rat_d$,
and we identify~$\Rat_d$ with an affine open subset of~$\PP^{2d+1}$ by
assigning the degree~$d$ rational function
\[
  \f_{\bfa,\bfb}(z) :=
  \frac{a_0z^d + a_1z^{d-1} + a_2z^{d-2} + \cdots + a_d}
       {b_0z^d + b_1z^{d-1} + b_2z^{d-2} + \cdots + b_d}
\]
to the point
\[
  [\bfa,\bfb] :=
  [a_0,a_1,\ldots,a_d,b_0,b_1,\ldots,b_d] \in \PP^{2d+1}.
\]
In this way~$\Rat_d$ is an affine scheme over~$\ZZ$, and for any
field~$K$, we may view~$K(z)$ as a disjoint union
\begin{equation}
  \label{eqn:KzcupRatdK}
  K(z) =  \bigcup_{d=0}^\infty \Rat_d(K) \subset \bigcup_{d=0}^\infty \PP^{2d+1}(K).
\end{equation}

However, we note that in general, the degree of $\gF_{m,k}(\f)(z)$ may
be strictly smaller than the degree of~$\f(z)$, so the Landen
transform $\gF_{m,k}:K(z)\to K(z)$ does not respect the disjoint union
decomposition~\eqref{eqn:KzcupRatdK}. For example, if~$\f(z)$ is a
polynomial, then $\deg\gF_{m,k}(\f)(z)\le\frac{1}{m}\deg\f(z)$;
see Section~\ref{section:Fmkasrationalmap}. Our main result describes the
rational self-maps of~$\PP^{2d+1}_\ZZ$ induced by the action
of~$\gF_{m,k}$ on a Zariski open subset of~$\Rat_d$.

\begin{theorem}
\label{theorem:RdmkFabmk}
Let $m\ge1$ and $0\le k<m$ be integers.
\begin{parts}
\Part{(a)}
For each $d\ge1$ there is a unique rational map
\[
  \gR_{d,m,k} : \PP^{2d+1}_\ZZ \dashrightarrow \PP^{2d+1}_\ZZ
\]
with the property that for all fields~$K$ in which $m\ne0$ and for all
degree~$d$ rational functions~$\f_{\bfa,\bfb}(z)\in\Rat_d(K)\subset
K(z)$ whose $(m,k)$-Landen transform satisfies
\[
  \deg_z \gF_{m,k}(\f_{\bfa,\bfb})(z) = d,
\]
we have
\[
  \gF_{m,k}(\f_{\bfa,\bfb})(z) = \gR_{d,m,k}\bigl([\bfa,\bfb]\bigr).
\]
\Part{(b)}
The indeterminacy locus of~$\gR_{d,m,k}$  is the linear subspace
\[
  \Zcal(\gR_{d,m,k})
  = \bigl\{[\bfa,\bfb]\in\PP^{2d+1}_\ZZ : \bfb=\bfzero \bigr\}
  \cong \PP^d,
\]
and the rational map~$\gR_{d,m,k}$ induces a \emph{morphism}
\[
  \gR_{d,m,k} : \PP^{2d+1}_\ZZ \setminus \Zcal(\gR_{d,m,k})
   \longrightarrow \PP^{2d+1}_\ZZ \setminus \Zcal(\gR_{d,m,k}).
\]
\Part{(c)}
The map $\gR_{d,m,0}:\PP^{2d+1}_\ZZ\dashrightarrow\PP^{2d+1}_\ZZ$ is a
dominant rational map of algebraic degree~$m$.
\Part{(d)}
For $1\le k<m$, the image of the rational map $\gR_{d,m,k}$ is the
hyperplane
\begin{equation}
  \label{eqn:hyperplanea00}
  \bigl\{[\bfa,\bfb]\in\PP^{2d+1}_\ZZ : a_0=0 \bigr\}
  \cong   \PP^{2d}_\ZZ.
\end{equation}
For all $0\le k<m$, the map~$\gR_{d,m,k}$ induces a dominant rational
map of algebraic degree~$m$ from the
hyperplane~\eqref{eqn:hyperplanea00} to itself.
\end{parts}
\end{theorem}

\begin{example} 
We consider the case~$d=2$ and~$m=2$.  Using the calculation given
later in Example~\ref{ex:gF20gF21}, we find that~$\gR_{2,2,0}$
and~$\gR_{2,2,1}$ are degree~$2$ rational maps
$\PP^5\dashrightarrow\PP^5$ given by the formul\ae
\begin{align*}
  \gR_{2,2,0} 
      &= [b_0a_0,\;b_2a_0-b_1a_1+b_0a_2,\;b_2a_2,
            \;b_0^2,\;2b_2b_0-b_1^2,\;b_2^2], \\
  \gR_{2,2,1} 
      &= [0,\;-b_1a_0+b_0a_1,\;b_2a_1-b_1a_2,\;b_0^2,\;2b_2b_0-b_1^2,\;b_2^2].
\end{align*}
As predicted by Theorem~\ref{theorem:RdmkFabmk}(b), both~$\gR_{2,2,0}$
and~$\gR_{2,2,1}$ have indeterminacy locus equal to
the~$2$-dimensional linear
subspace~$\bigl\{[\bfa,\bfzero]\bigr\}\subset\PP^5$.  One can check
that it requires more than simply blowing up~$\PP^5$ along this
subspace in order to make~$\gR_{2,2,0}$ and~$\gR_{2,2,1}$ into
morphisms.
\end{example}

Taking $k=0$ leads to interesting families of commuting maps.  (See
Proposition~\ref{prop:compositionlaw} for general composition
properties of~$\gR_{m,k,d}$.)

\begin{corollary}
\label{corollary:Rdm0commutingfamily}
Fix a degree $d$. Then
\[
  \{ \gR_{d,m,0} : m=1,2,3,\ldots \}
\]
is a set of \emph{commuting dominant} rational endomorphisms
of~$\PP^{2d+1}_\ZZ$ of algebraic degree~$m$. More precisely,
dehomogenizing and specializing, the maps~$\gR_{d,m,0}$ induce
commuting dominant polynomial endomorphisms of the affine linear
subspaces
\[
  \bigl\{[\bfa,\bfb]\in\PP^{2d+1} : b_0\ne0\bigr\} \cong \AA^{2d+1}
\]
and 
\[
   \bigl\{[\bfa,\bfb]\in\PP^{2d+1} :
     a_0=0~\text{and}~b_0\ne0\bigr\} \cong \AA^{2d} .
\]
\end{corollary}

\begin{remark}
The classification of commuting rational maps in one variable was
solved by Ritt~\cite{MR1501252} in the 1920s. More recently, there has
been some work on classifying commuting \emph{endomorphisms} of
$\PP^n$~\cite{MR1824960,MR1931758}, as well as various papers,
including~\cite{MR1871293,MR1842291}, that study higher dimensional
Latt\`es maps, and work on commuting birational self-maps of~$\PP^2$
(and more generally of a compact K\"ahler
surface)~\cite{MR2811600}. But there seem to be few non-trivial
examples known of commuting \emph{rational} (non-birational)
self-maps of~$\PP^n$, and as far as we are aware, the family of commuting
Landen maps described in Corollary~\ref{corollary:Rdm0commutingfamily}
has not previously been studied.
\end{remark}

\begin{remark}
If we treat rational maps of degree~$d-1$ as degenerate maps of degree~$d$, 
we obtain a natural embedding
\begin{align*}
  \iota_{d-1,d} : \PP^{2d-1}&\longrightarrow\PP^{2d+1},\\
  [a_0,\ldots,a_{d-1},b_0,\ldots,b_{d-1}]
  &\longmapsto
  [0,a_0,\ldots,a_{d-1},0,b_0,\ldots,b_{d-1}].
\end{align*}
Then the maps in Theorem~\ref{theorem:RdmkFabmk} fit together via
\[
  \gR_{d,m,k}\circ \iota_{d-1,d} = \iota_{d-1,d}\circ\gR_{d-1,m,k}.
\]
\end{remark}

\begin{remark}
Since~$\PP^{2d+1}_\ZZ$ is smooth, the rational function~$\gR_{d,m,k}$
is defined off of a codimension~$2$
subscheme. (Theorem~\ref{theorem:RdmkFabmk}(b) says that in fact, the
indeterminacy locus has codimension~$d+1$.) In
particular,~$\gR_{d,m,k}$ induces a rational map on every special fiber
$\PP^{2d+1}_{\FF_p}\dashrightarrow\PP^{2d+1}_{\FF_p}$, even if~$p\mid
m$, despite the~$\frac{1}{m}$ factor appearing in the
formula~\eqref{eqn:Fmkfwmdef} defining~$\gF_{m,k}$.
\end{remark}

We conclude the introduction by summarizing the contents of this
article. Section~\ref{section:history} briefly describes some of the
history and uses of the Landen
transformation. Section~\ref{section:examples} illustrates the Landen
transform by giving explicit formulas for~$\gF_{m,k}(\f)$ when~$\f$
has degree~$2$ and~$3$ and~$m$ equals~$2$
and~$3$. Section~\ref{section:laurent} and~\ref{section:composition}
give, respectively, the effect of~$\gF_{m,k}$ on formal Laurent series
and an elementary composition formula for~$\gF_{m,k}$. In
Section~\ref{section:Fmkasquotient} we prove a key proposition that
writes~$\gF_{m,k}(\f)(z)$ as a quotient of
polynomials~$G_{\bfa,\bfb,m,k}(z)/H_{\bfb,m}(z)$ whose coefficients
are $\ZZ$-integral polynomials in the coefficients of~$\f$, and we
describe various properties of~$G_{\bfa,\bfb,m,k}(z)$
and~$H_{\bfb,m}(z)$. This material is used in
Section~\ref{section:Fmkasrationalmap} to prove our main result
(Theorem~\ref{theorem:RdmkFabmk}). We conclude in
Section~\ref{section:mapinducedbyHbm} by showing that the coefficients
of the denominator~$H_{\bfb,m}(z)$ of the Landen transformation
induces a morphism~$\PP^d\to\PP^d$ that is birationally conjugate to
the $m$'th power map.

\begin{acknowledgement}
We would like to thank Michael Rosen for his assistance and for
pointing out the identification described in
Corollary~\ref{corollary:affinemonoid}, and Doug Lind and Klaus Schmidt for
suggesting the material in Section~\ref{section:mapinducedbyHbm}. The
second and third authors would also like thank Tzu-Yueh (Julie) Wang
and Liang-Chung Hsia for inviting them to participate in a Conference
on Diophantine Problems and Arithmetic Dynamics held at Academia
Sinica, Taipei, June 2013.
\end{acknowledgement}

\section{History and applications}
\label{section:history}
Let~$K$ be a field that is not of characteristic~$2$.  The classical
Landen transformation is the map~$\gF=\gF_{2,1}$ on the space of rational
functions~$K(z)$ given by the formula
\[
  \gF(\f)(z) = \frac{\f(\sqrt z\,)-\f(-\sqrt z\,)}{2\sqrt z}.
\]
When~$K=\RR$ or~$\CC$, the Landen transformation can be used to
numerically compute the integral~$\int_0^\infty\f(z)\,dz$ for certain
choices of the rational function~$\f$. More precisely, for
appropriate~$\f$ one shows that
\[
  \int_0^\infty\f(z)\,dz =   \int_0^\infty\gF(\f)(z)\,dz
\]
and then studies the dynamics of~$\gF$, i.e., the behavior of the
orbit $(\gF^n(\f))_{n\ge1}$ of the rational map~$\f$ under iteration
of the transformation~$\gF$. See~\cite{%
MR1980323,
MR2210638,
MR1771260,
MR1885619,
MR2220755,
MR1960939,
MaMo2,
MaMo3,
MR2727063,
MR1522656}
for work in this area, as well as~\cite{MaMo1} for a survey of the
theory of Landen transformations.
\par
The origins of the subject go back to Landen's work~\cite{landen1,landen2} 
on iterative methods to compute certain integrals. The method was
rediscovered and extended by Gauss~\cite{gauss}
and is often referred to as Gauss's Arithmetic-Geometric Mean (AGM) method.
\par
The authors of~\cite{MR2210638} also point to a related transformation
\[
  \gC(\f)(z) = \gF_{2,0}(\f)(z) = \frac{\f(\sqrt z\,)+\f(-\sqrt z\,)}{2}
\]
whose dynamics is analyzed in~\cite{BJMMM}, and they indicate that
there are natural generalizations to higher degree transformations
that they plan to study in a future work.  These higher degree
transformations are the maps~$\gF_{m,k}$ described in
Proposition~\ref{proposition:Gphicharac}.

\begin{remark} 
The formula~\eqref{eqn:Fmkfwmdef} for~$\gF_{m,k}$ is given
in~\cite{MR2249990}, where the author determines a basis for the set
of rational functions~$\f(z)\in\CC(z)$ that are fixed by~$\gF_{m,k}$.
\end{remark}

\begin{remark}
We observe that if we view~$K(z)$ as a $K$-vector space,
then~$\gF_{m,k}$ is clearly a $K$-linear transformation of~$K(z)$.
However, when we view~$K(z)$ as a field, the action of~$\gF_{m,k}$ is
more complicated.
\end{remark}

\begin{proposition} 
\label{proposition:FphiinKwm}
Let $m\ge1$ and $k\in\ZZ$. Then for all~$\f(z)\in K(z)$,  the
expression~\eqref{eqn:Fmkfwmdef}
\begin{equation}
  \label{eqn:gFw}
   \frac{1}{mw^k}\sum_{t=0}^{m-1} \z_m^{-kt}\f(\z_m^t w)
\end{equation}
appearing in Proposition~$\ref{proposition:Gphicharac}$ is 
in~$K(w^m)$. Further, it is independent of the choice of a particular
primitive~$m^{\text{th}}$~root of unity~$\z_m$.
\end{proposition} 
\begin{proof}
Let~$\Kbar/K$ be an algebraic closure of~$K$.  The field extension
$\Kbar(w)/\Kbar(w^m)$ is a Kummer extension whose Galois group is
cyclic and generated by the automorphism~$w\to\z w$. (As always, we
are assuming that~$m$ is prime to the characteristic of~$K$.)  But it
is easy to check that the expression~\eqref{eqn:gFw} is invariant
under the substitution~$w\to\z w$. Hence it is in~$\Kbar(w^m)$, and
indeed in~$K(\z_m)(w^m)$. Next we observe that~\eqref{eqn:gFw} is also
invariant under an element~$\s$ in the Galois group
of~$K(\z_m)(w^m)/K(w^m)$, since the effect of such an element is to
send~$\z_m$ to~$\z_m^j$ for some~$j$ satisfying $\gcd(j,m)=1$, so it
simply rearranges the terms in the sum.  This proves
that~\eqref{eqn:gFw} is in~$K(w^m)$, and also shows
that~\eqref{eqn:gFw} does not depend on the choice of~$\z_m$.
\end{proof}

\section{Examples}
\label{section:examples}
We compute some examples of Landen transforms for generic rational
maps of degrees~$2$ and~$3$, i.e., we give explicit formulas for the
rational maps $\gR_{d,m,k}:\PP^{2d+1}\dashrightarrow\PP^{2d+1}$
for small values of~$d,m,k$.

\begin{example}
\label{ex:gF20gF21}
Consider the generic rational map of degree~$2$,
\[
  \f(z) = \frac{a_0z^2+a_1z+a_2}{b_0z^2+b_1z+b_2}.
\]
A simple calculation shows that the two transformations~$\gF_{2,0}$
and~$\gF_{2,1}$ are given by
\begin{align*}
  \gF_{2,0}(\f)(z) &= \frac{b_0a_0z^2 +(b_2a_0-b_1a_1+b_0a_2)z+b_2a_2}
        {b_0^2z^2+(2b_2b_0-b_1^2)z+b_2^2}, \\
  \gF_{2,1}(\f)(z) &= \frac{(-b_1a_0+b_0a_1)z+b_2a_1-b_1a_2}
        {b_0^2z^2+(2b_2b_0-b_1^2)z+b_2^2}, 
\end{align*}
and
\begin{align*}
\gF_{3,0}(\f)(z) &=
    \frac{b_0^2 a_0 z^2 + (-b_2 b_1 a_0 - b_2 b_0 a_1 + b_1^2
       a_1 - b_1 b_0 a_2) z + b_2^2 a_2}
      {b_0^3 z^2 + (-3 b_2 b_1 b_0 + b_1^3 ) z + b_2^3 }, \\
\gF_{3,1}(\f)(z) &=
    \frac{(-b_2 b_0 a_0 + b_1^2 a_0 - b_1 b_0 a_1 + b_0^2 a_2) z +
      (b_2^2 a_1 - b_2 b_1 a_2) }
    { b_0^3 z^2 + (-3 b_2 b_1 b_0 + b_1^3) z + b_2^3 }, \\
\gF_{3,2}(\f)(z) &=
    \frac{ (-b_1 b_0 a_0 + b_0^2 a_1) z + (b_2^2 a_0 - b_2 b_1 a_1 -
      b_2 b_0 a_2 + b_1^2 a_2) }
     { b_0^3 z^2 + (-3 b_2 b_1 b_0 + b_1^3) z + b_2^3 }.
\end{align*}
\end{example}

\begin{example}
\label{ex:gF30gF31gF32}
Similarly, for a generic rational map of degree~$3$,
\[
  \f(z) = \frac{a_0z^3+a_1z^2+a_2z+a_3}{b_0z^3+b_1z^2+b_2x+b_3},
\]
the first few Landen transforms act by
{\small
\begin{align*}
\gF_{2,0} &=
    \frac{b_0a_0z^3 + (b_2a_0 - b_1a_1 + b_0a_2)z^2 + (-b_3a_1
      + b_2a_2 - a_3b_1)z - b_3a_3}
      {b_0^2z^3 + (2b_2b_0 - b_1^2)z^2 + (-2b_3b_1 + b_2^2 )z -
        b_3^2}, \\
\gF_{2,1} &=
       \frac{ (-b_1a_0 + b_0a_1)z^2 + (-b_3a_0 + b_2a_1 - b_1a_2
         + a_3b_0)z + (-b _3a_2 + a_3b_2)}
      {b_0^2z^3 + (2b_2b_0 - b_1^2)z^2 + (-2b_3b_1 + b_2^2)z -
        b_3^2},
\end{align*}%
}
and
{%
\begin{align*}
\gF_{3,0} &= \tfrac{
   \genfrac{}{}{0pt}{1}{
      b_0^2 a_0 z^3 
      + (2 b_3 b_0 a_0 - b_2 b_1 a_0 - b_2 b_0 a_1 + b_1^2 a_1 - b_1
	       b_0 a_2 + b_0^2 a_3) z^2 
      \hspace{1.3in}\hfill}
      {\hspace{1.3in}\hfill
      + (b_3^2 a_0 - b_3 b_2 a_1 - b_3 b_1 a_2 + b_2^2 a_2
		+ 2 b_3 b_0 a_3 - b_2 b_1 a_3) z  
      + b_3^2 a_3 
      }
   }
   {
  b_0^3 z^3 
       + (3 b_3 b_0^2 - 3 b_2 b_1 b_0 + b_1^3) z^2 
       + (3 b_3^2 b_0 - 3 b_3 b_2 b_1 + b_2^3) z 
       + b_3^3
  },\\
\gF_{3,1} &= \tfrac{
   \genfrac{}{}{0pt}{1}{
      (-b_2 b_0 a_0 + b_1^2 a_0 - b_1 b_0 a_1 + b_0^2 a_2) z^2 
      \hspace{1.3in}\hfill}
      {\hspace{.3in}\hfill
      + (-b_3 b_2 a_0 - b_3 b_1 a_1 + b_2^2 a_1  + 2 b_3 b_0 a_2 - b_2 b_1 a_2
	  - b_2 b_0 a_3 + b_1^2 a_3) z 
      + b_3^2 a_2 - b_3 b_2 a_3
      }
   }
   {
  b_0^3 z^3 
       + (3 b_3 b_0^2 - 3 b_2 b_1 b_0 + b_1^3) z^2 
       + (3 b_3^2 b_0 - 3 b_3 b_2 b_1 + b_2^3) z 
       + b_3^3
  },\\
\gF_{3,2} &= \tfrac{
   \genfrac{}{}{0pt}{1}{
     (-b_1 b_0 a_0 + b_0^2 a_1) z^2 
     + (-b_3 b_1 a_0 + b_2^2 a_0 + 2 b_3 b_0 a_1 - b_2 b_1 a_1
	 - b_2 b_0 a_2 + b_1^2 a_2 - b_1 b_0 a_3) z 
      \hspace{.3in}\hfill}
      {\hspace{.3in}\hfill
     + b_3^2 a_1 - b_3 b_2 a_2 - b_3 b_1 a_3 + b_2^2 a_3
    }
  }
  {
  b_0^3 z^3 
       + (3 b_3 b_0^2 - 3 b_2 b_1 b_0 + b_1^3) z^2 
       + (3 b_3^2 b_0 - 3 b_3 b_2 b_1 + b_2^3) z 
       + b_3^3
  }.\\
\end{align*}
}%
\end{example}

\section{The effect of $\gF_{m,k}$ on Laurent series}
\label{section:laurent}
An elementary calculation reveals the effect of~$\gF_{m,k}$ on
a Laurent series around~$0$, and in particular on the series associated to
a rational function.

\begin{proposition}
\label{prop:gFeffectonseries}
Let
\[
  \f(z) = \sum_{n\in\ZZ} a_nz^n
\]
be a \textup(formal\textup) Laurent series. Then
\[
  \gF_{m,k}(\f)(z) 
   =  \sum_{j\in\ZZ} a_{mj+k}z^j.
\]
\end{proposition}
\begin{proof}
We compute
\begin{align*}
  \gF_{m,k}(\f)(w^m) 
  &= \frac{1}{mw^k}
     \sum_{t=0}^{m-1} \z_m^{-kt}\f(\z_m^t w) \\
  &= \frac{1}{mw^k}
     \sum_{t=0}^{m-1} \z_m^{-kt}\sum_{n\in\ZZ} a_n(\z_m^t w)^n \\
  &= \sum_{n\in\ZZ} a_nw^{n-k}
    \biggl(\frac{1}{m} \sum_{t=0}^{m-1} \z_m^{(n-k)t} \biggr)  \\
  & =  \sum_{\substack{n\in\ZZ\\n\equiv k\MOD{m}\\}} a_nw^{n-k}.
\end{align*}
This completes the proof of Proposition~\ref{prop:gFeffectonseries}.
\end{proof}

\section{Composition of $\gF_{m,k}$ operators}
\label{section:composition}
The transformations~$\gF_{m,k}$ and~$\gF_{n,\ell}$ do not generally
commute, but they do if $k(n-1)=\ell(m-1)$. In particular,
if~$k=\ell=0$, then they commute for all values of~$m$ and~$n$. The
next elementary result gives a general composition formula.

\begin{proposition}
\label{prop:compositionlaw}
For all~$m,n\ge1$ and all $k,\ell\in\ZZ$,
\[
  \gF_{m,k}\circ\gF_{n,\ell} = \gF_{mn,kn+\ell}.
\]
In particular, the $r$'th iterate of~$\gF_{m,k}$ is given by
\[
  \gF_{m,k}^r = \gF_{m^r,(m^{r-1}+m^{r-2}+\cdots+m+1)k}.
\]
\end{proposition}
\begin{proof}
We choose our roots of unity to satisfy~$\z_m=\z_{mn}^n$
and~$\z_n=\z_{mn}^m$. To ease the computation, we let
$z=w^m$ and $w=u^n$, so $z=u^{mn}$. Then
\begin{align*}
  \gF_{m,k}\bigl(\gF_{n,\ell}(\f)\bigr)(z)
  &= \frac{1}{mu^{nk}}\sum_{t=0}^{m-1} \z_m^{-kt}
      \cdot \gF_{n,\ell}(\f)(\z_m^t w) \\
  &= \frac{1}{mu^{nk}}\sum_{t=0}^{m-1} \z_m^{-kt}
      \biggl(\frac{1}{n(\z_m^{t/n}u)^\ell}
        \sum_{s=0}^{n-1}\z_n^{-\ell s}\f(\z_n^s\z_m^{t/n}u)\biggr) \\
  &= \frac{1}{mnu^{nk+\ell}} \sum_{t=0}^{m-1}\sum_{s=0}^{n-1}
      \z_{mn}^{-(nkt+\ell(ms+t))} \f(\z_{mn}^{ms+t}u) \\
  &\omit\hfill \qquad\text{we let $i=ms+t$,}\\
  &= \frac{1}{mnu^{nk+\ell}} \sum_{i=0}^{mn-1}
      \z_{mn}^{-(nk(i\bmod m) + \ell i)} \f(\z_{mn}^i u) \\
  &= \frac{1}{mnu^{nk+\ell}} \sum_{i=0}^{mn-1}
      \z_{mn}^{-(nk + \ell)i} \f(\z_{mn}^i u) \\
  &=\gF_{mn,nk+\ell}(\f)(z).
\end{align*}
This proves the first formula, and the second follows by induction.
\end{proof}

Proposition~\ref{prop:compositionlaw} allows us to describe the
monoid of~$\gF_{m,k}$~operators in terms of a matrix monoid.

\begin{corollary}
\label{corollary:affinemonoid}
Let~$M$ be the monoid of integral matrices
\[
  M = \left\{
     \begin{pmatrix}
	 m & 0 \\ k & 1 \\
     \end{pmatrix} : m,k\in\ZZ,\;m\ge1
   \right\}
\]
under matrix multiplication. Then the map
\begin{equation}
  \label{eqn:monoidisom}
  M \longrightarrow \{\gF_{m,k} : m,k\in\ZZ,\;m\ge1 \},
  \qquad
  \begin{pmatrix} m&0\\k&1\\ \end{pmatrix}
  \longmapsto \gF_{m,k},
\end{equation}
is a monoid isomorphism.
\end{corollary}
\begin{proof}
The map~\eqref{eqn:monoidisom} is clearly surjective, while
Proposition~\ref{prop:compositionlaw} and the matrix multiplication
$
  \left(\begin{smallmatrix}
      m & 0 \\ k & 1 \\
  \end{smallmatrix}\right)
  \left(\begin{smallmatrix}
      n & 0 \\ \ell & 1 \\
  \end{smallmatrix}\right)
  =
  \left(\begin{smallmatrix}
      mn & 0 \\ kn+\ell & 1 \\
  \end{smallmatrix}\right)
$
shows that~\eqref{eqn:monoidisom} is a monoid homomorphism.  For
injectivity, we suppose that~$\gF_{m,k}=\gF_{n,\ell}$.
Proposition~\ref{prop:gFeffectonseries} tells us that $\gF_{m,k}(z^d)$
is equal to $z^{(d-k)/m}$ if $d\equiv k\pmod{m}$, and equal to~$0$
otherwise.  Taking $d=k+me$, our assumption
that~$\gF_{m,k}=\gF_{n,\ell}$ implies that
\[
  z^e = \gF_{m,k}(z^{me+k}) = \gF_{n,\ell}(z^{me+k}) 
  = z^{(me+k-\ell)/n},
\]
where necessarily~$n$ divides~$me+k-\ell$. Equating the exponents, 
we have
\[
  (n-m)e =k-\ell
  \quad\text{for all $e\in\ZZ$.}
\]
The right-hand side is independent of~$e$, and hence we must
have~$n=m$ and $k=\ell$, which concludes the proof
that~\eqref{eqn:monoidisom} is injective.
\end{proof}

\section{Writing $\gF_{m,k}$ as a quotient of integral polynomials}
\label{section:Fmkasquotient}
Our primary goal in this section is to write~$\gF_{m,k}(\f)$, for a
generic rational function~$\f$ of degree~$d$, as a quotient of
$\ZZ$-integral polynomials in~$z$ and the coefficients of~$\f$. We
recall from the introduction that we are identifying the
space~$\Rat_d$ of rational functions of degree~$d$ with an affine open
subset of~$\PP^{2d+1}$.  More precisely, for a $(d+1)$-tuple
$\bfa=[a_0,\ldots,a_d]$, we let
\[
  F_\bfa(X,Y) = a_0X^d + a_1X^{d-1}Y + a_2X^{d-2}Y^2 + 
   \cdots + a_dY^d,
\]
and we associate to each point~$[\bfa,\bfb]\in\PP^{2d+1}$ the rational
map
\[
  \f_{\bfa,\bfb}:\PP^1\longrightarrow\PP^1,
  \qquad
  \f_{\bfa,\bfb}\bigl([X,Y]\bigr)=\bigl[F_\bfa(X,Y),F_\bfb(X,Y)\bigr].
\]
Then~$\Rat_d$ is the complement of the resultant hypersurface
\[
  \Rat_d = \bigl\{[\bfa,\bfb]\in\PP^{2d+1} :
  \Resultant(F_\bfa,F_\bfb)\ne0 \bigr\}.
\]
In order to emphasize this inclusion, we let
\[
  \Ratbar_d\cong\PP^{2d+1}.
\]
Points~$[\bfa,\bfb]\in\Ratbar_d\setminus\Rat_d$ correspond to rational
maps~$\f_{\bfa,\bfb}$ of lower degree, but we note that different
points in~$\Ratbar_d\setminus\Rat_d$ may correspond to the the same
rational map. (For a discussion of~$\Rat_d$ and its various
extensions, quotients, and compactifications, see for
example~\cite[Section~4.3]{MR2316407} or~\cite{MR2884382}.)
\par
It is often convenient to dehomogenize~$z=X/Y$, so by abuse
of notation we will write
\[
  \f_{\bfa,\bfb}(z) 
  = \frac{F_\bfa(z)}{F_\bfb(z)}
  = \frac{F_\bfa(z,1)}{F_\bfb(z,1)}
\]
for the associated rational function and its dehomogenized numerator and
denominator.

\begin{proposition} 
\label{prop:resultantGH}
Let $m\ge1$ and $0\le k<m$.
\begin{parts}
\Part{(a)}
There are unique polynomials
\[
  G_{\bfa,\bfb,m,k}(z) \in \ZZ[\bfa,\bfb,z]
  \quad\text{and}\quad
  H_{\bfb,m}(z) \in \ZZ[\bfb,z]
\]
satisfying
\begin{align}
  \label{eqn:Gabnkwm}
  G_{\bfa,\bfb,m,k}(w^m)
  &= \frac{1}{mw^k}\sum_{t=0}^{m-1} \z_m^{-kt}F_\bfa(\z_m^t w)
           \prod_{s\ne t} F_\bfb(\z_m^s w),\\
  \label{eqn:Hbmwm}
  H_{\bfb,m}(w^m)
   &=\prod_{t=0}^{m-1} F_\bfb(\z_m^t w).
\end{align}
\Part{(b)}
Let~$\f_{\bfa,\bfb}(z)=F_\bfa(z)/F_\bfb(z)$. Then
\[
  \gF_{m,k}(\f_{\bfa,\bfb})(z) 
  = \frac{G_{\bfa,\bfb,m,k}(z)}{H_{\bfb,m}(z)}.
\]
\Part{(c)}
The polynomials $G_{\bfa,\bfb,m,k}(z)$ and $H_{\bfb,m}$ have the
following homogeneity properties\textup:
\begin{parts}
\Part{(i)}
The~$z$-coefficients of~$H_{\bfb,m}(z)$, considered as elements of~$\ZZ[\bfb]$,
are homogeneous of degree~$m$. 
\Part{(ii)}
The~$z$-coefficients of~$G_{\bfa,\bfb,m,k}(z)$, considered as elements
of $\ZZ[\bfa,\bfb]$, are bi-homogeneous of bi-degree~$(1,m-1)$.
\Part{(iii)}
If we make~$\ZZ[\bfa,\bfb,z]$ into a graded $\ZZ$-algebra by assigning
weights
\begin{equation}
  \label{eqn:weights}
  \wt(z)=m\quad\text{and}\quad\wt(a_i)=\wt(b_i)=i,
\end{equation}
then $G_{\bfa,\bfb,m,k}(z)$ and $H_{\bfb,m}(z)$ are weight homogeneous
with weights
\[
  \wt(G_{\bfa,\bfb,m,k}(z))=md-k
  \quad\text{and}\quad \wt(H_{\bfb,m}(z))=md.
\]
\end{parts}
\Part{(d)}
The polynomials~$G_{\bfa,\bfb,m,0}(z)$ and~$H_{\bfb,m}(z)$ have the form
\begin{align*}
  G_{\bfa,\bfb,m,0}(z) &= (-1)^{(m+1)d}a_0b_0^{m-1}z^d + O(z^{d-1}), \\
  G_{\bfa,\bfb,m,k}(z) &= O(z^{d-1})\quad\text{for $1\le k<m$,} \\
  H_{\bfb,m}(z) &= (-1)^{(m+1)d}b_0^mz^d + O(z^{d-1}).
\end{align*}
In particular, both~$G_{\bfa,\bfb,m,0}(z)$ and~$H_{\bfb,m}(z)$
have~$z$-degree~$d$, while $G_{\bfa,\bfb,m,k}(z)$ has~$z$-degree
strictly smaller than~$d$ for~$1\le k<m$\textup{;}
cf.\ Remark~$\ref{remark:monsinGabmk}$.
\Part{(e)}
Let
\[
  F_\bfb(z) = b_0 \prod_{i=1}^d (z-\b_i)
\]
be the factorization of~$F_\bfb(z)$ in some integral extension
of~$\ZZ[\bfb,z]$. Then
\[
  H_{\bfb,m}(z) = (-1)^{(m+1)d} b_0^m \prod_{i=1}^d (z-\b_i^m).
\]
\Part{(f)}
We have
\begin{align}
  \label{eqn:resGHxi}
  \Resultant(G_{\bfa,\bfb,m,k},&H_{\bfb,m}) \notag\\
  &= (-1)^{d(m-k+1)} b_0^{m-1}b_d^{m-1-k} 
        \Resultant(F_\bfa,F_\bfb)\frac{\Disc(H_{\bfb,m})}{\Disc(F_\bfb)}.
\end{align}
\Part{(g)}
We have
\[
  \frac{\Disc(H_{\bfb,m})}{\Disc(F_\bfb)}
  \in \ZZ[b_0,\ldots,b_d],
\]
i.e., the polynomial ${\Disc(F_\bfb)}$ divides the polynomial
${\Disc(H_{\bfb,m})}$ in~$\ZZ[\bfb]$.
\Part{(h)}
Let~$\b_1,\ldots,\b_d$ be the roots of~$F_\bfb(z)$ as in~\textup{(e)},
and assume that $m\ge2$. 
The quotient ${\Disc(H_{\bfb,m})}/{\Disc(F_\bfb)}$ vanishes if and
only if there is a pair of indices~$i\ne j$ such that either\textup:
\[
  \begin{array}{r@{\quad}l}
    \textup{(i)}&\b_i=\b_j=0. \\
    \textup{(ii)}&\b_i\ne\b_j\quad\text{and}\quad \b_i^m=\b_j^m.\\
  \end{array}
\]
\end{parts}
\end{proposition}

\begin{remark}
Note that the resultant formula~\eqref{eqn:resGHxi} implicitly assumes
that~$F_\bfa$ and~$F_\bfb$ have degree~$d$. In other words, they 
should first be homogenenized to be polynomials of degree~$d$,
then the polynomials~$G_{\bfa,\bfb,m,k}$ and~$H_{\bfb,m}$ are also
homogeneous of degree~$d$ and the resultant is calculated accordingly.
With this convention, we see that
\[
  \deg_z\bigl(\gF_{m,k}(\f_{\bfa,\bfb})\bigr)
  = \deg_z(\f_{\bfa,\bfb})
  \quad\Longleftrightarrow\quad
  \Resultant(G_{\bfa,\bfb,m,k},H_{\bfb,m})\ne 0.
\]
Thus Proposition~\ref{prop:resultantGH}(f) can be used to answer the
question of whether the operator~$\gF_{m,k}$ preserves the degree
of~$\f(z)$.
\end{remark}

\begin{example}
\label{ex:resultantGH}
We illustrate Proposition~\ref{prop:resultantGH} for~$d=2$ and~$m=2$. 
We have
\begin{align*}
G_{\bfa,\bfb,2,0} &= b_0a_0z^2 + (b_2a_0-b_1a_1+b_0a_2)z+b_2a_2,\\
G_{\bfa,\bfb,2,1} &= (-b_1a_0+b_0a_1)z+b_2a_1-b_1a_2,\\
H_{\bfb,2} &= b_0^2z^2+(2b_2b_0-b_1^2)z+b_2^2,
\end{align*}
from Example~\ref{ex:gF20gF21}. The resultant of
the quadratic polynomials~$F_\bfa$ and~$F_\bfb$
is given by a well-known formula~\cite[\S27]{vanderwaerden:algebra},
while
\[
  \Disc(H_{\bfb,2}) = b_1^2(-4b_2b_0 + b_1^2) = b_1^2\Disc(F_\bfb).
\]
Then formulas for $\Resultant(G_{\bfa,\bfb,2,0},H_{\bfb,2})$ and
$\Resultant(G_{\bfa,\bfb,2,1},H_{\bfb,2})$ can be derived using
Proposition~\ref{prop:resultantGH}(f).
\end{example}

\begin{remark}
\label{remark:monsinGabmk}
It is possible to compute some of the other monomials appearing 
in $G_{\bfa,\bfb,m,k}(z)$ and $H_{\bfb,m}$ by evaluating more complicated 
sums and products of powers of roots of unity. For example,
as an element of~$\ZZ[\bfa,\bfb,z]$, the polynomial~$G_{\bfa,\bfb,m,k}(z)$
contains the monomials 
\[
   (-1)^{(m+1)d}a_{m-k}b_0^{m-1}z^{d-1}
   \quad\text{and}\quad
   (-1)^{(m+1)d+1}a_{m-k-1}b_0^{m-2}b_1z^{d-1}.
\]
In particular, if~$1\le k<m$, then $G_{\bfa,\bfb,m,k}(z)$
has~$z$-degree equal to~$d-1$. We omit the proof.
\end{remark}

Since it will come up frequently, we record here the elementary
fact
\begin{equation}
  \label{eqn:prodzmtneg1m1}
  \prod_{t=0}^{m-1} \z_m^t = (-1)^{m+1}.
\end{equation}

\begin{proof}[Proof of Proposition $\ref{prop:resultantGH}$]
For the moment, we let
\[
  g(w) = w^{m-k} \sum_{t=0}^{m-1} \z_m^{-kt}F_\bfa(\z_m^t w)
           \prod_{s\ne t} F_\bfb(\z_m^s w),
  \quad
  h(w) = \prod_{t=0}^{m-1} F_\bfb(\z_m^t w).
\]
It is clear from these formulas that $g(w)\in\ZZ[\z_m][\bfa,\bfb,w]$
and $h(w)\in\ZZ[\z_m][\bfb,w]$. We claim that
\begin{equation}
  \label{eqn:ghzmwghw}
  g(\z_m w)=g(w)\quad\text{and}\quad h(\z_m w)=h(w).
\end{equation}
This is clear for~$h$, while for~$g$ we compute
\begin{align*}
  g(\z_m w) 
  & = (\z_m w)^{m-k}  \sum_{t=0}^{m-1} \z_m^{-kt}F_\bfa(\z_m^{t+1} w)
           \prod_{s\ne t} F_\bfb(\z_m^{s+1} w) \\ 
  & = \z_m^{-k} w^{m-k}  \sum_{t=1}^{m} \z_m^{-k(t-1)}F_\bfa(\z_m^{t} w)
           \prod_{s\ne t} F_\bfb(\z_m^{s} w) \\ 
  &= g(w).
\end{align*}
It follows from~\eqref{eqn:ghzmwghw} that
$g(w)\in\ZZ[\z_m][\bfa,\bfb,w^m]$ and
$h(w)\in\ZZ[\z_m][\bfb,w^m]$. (Note that~$\z_m$ is a primitive $m$'th
root of unity.)
\par
It is also clear from the formulas for~$g$ and~$h$
that they do not change if we replace~$\z_m$ by any other
primitive $m$'th root of unity. Hence their coefficients are fixed
by~$\Gal\bigl(\QQ(\z_m)/\QQ\bigr)$, which shows that
\[
  g(w)\in\ZZ[\bfa,\bfb,w^m]
  \quad\text{and}\quad
  h(w)\in\ZZ[\bfb,w^m].
\]
This completes the proof that~$H_{\bfb,m}(z)=h(z^{1/m})$ is
in~$\ZZ[\bfb,z]$. Further, in order to show that
\[
  G_{\bfa,\bfb,m,k}(z)=\frac{1}{mz}g(z^{1/m}) \in\ZZ[\bfa,\bfb,z],
\]
it remains only to prove that~$g(w)$ is in the ideal
$mw^m\ZZ[\bfa,\bfb,w^m]$.
\par
The definition of~$g(w)$ shows that~$g(w)$ is a multiple of~$w^{m-k}$,
and $m-k>0$ by assumption, so~$g(0)=0$. Since we also know that~$g(w)$
is a polynomial in~$w^m$, it follows that~$g(w)$ is a multiple of~$w^m$,
i.e., $g(w)\in w^m\ZZ[\bfa,\bfb,w^m]$.
\par
We next prove~(b), after which we will complete the proof of~(a). 
Thus
\begin{align*}
  \gF_{m,k}(\f_{\bfa,\bfb})(z) 
  &= \frac{1}{mw^k} \sum_{t=0}^{m-1} \z_m^{-kt}\f_{\bfa,\bfb}(\z_m^t w) \\
  &= \frac{1}{mw^k} \sum_{t=0}^{m-1} \z_m^{-kt}
        \frac{F_\bfa(\z_m^t w)}{F_\bfb(\z_m^t w)} \\
  &= \frac{\displaystyle
           \frac{1}{mw^k}\sum_{t=0}^{m-1} \z_m^{-kt}F_\bfa(\z_m^t w)
           \prod_{s\ne t} F_\bfb(\z_m^s w) }
        {\displaystyle\prod_{t=0}^{m-1} F_\bfb(\z_m^t w) } \\
  &= \frac{G_{\bfa,\bfb,m,k}(z)}{H_{\bfb,m}(z)},
\end{align*}
which gives~(b).
\par
We observe that we can expand $F_\bfb(w)^{-1}$ as a Laurent series
\[
  \frac{1}{F_\bfb(w)}
  = \sum_{j=j_0}^\infty c_j(\bfb)w^j
  \quad\text{with $c_j(\bfb)\in\ZZ[\bfb,b_d^{-1}]$.}
\]
This allows us to expand~$g(w)$ as
\begin{align*}
  g(w) 
  &= mw^m G_{\bfa,\bfb,m,k}(w^m) 
     \quad\text{by definition,}\\
  &= mw^m \gF_{m,k}(\f_{\bfa,\bfb})(w^m) H_{\bfb,m}(w^m) 
      \quad\text{from (b),}\\
  &= w^{m-k} \left(\sum_{t=0}^{m-1} \z_m^{-kt}\f_{\bfa,\bfb}(\z_m^t w)\right)
      H_{\bfb,m}(w^m) 
     \quad\text{by definition,}\\
  &= w^{m-k} \left(\sum_{t=0}^{m-1} 
       \z_m^{-kt} \frac{F_\bfa(\z_m^t w)}{F_\bfb(\z_m^t w)}\right)
      H_{\bfb,m}(w^m)  \\
  &= w^{m-k} \left(\sum_{t=0}^{m-1} 
       \z_m^{-kt} \sum_{i=0}^d  a_{d-i}(\z_m^t w)^i
                      \sum_{j=j_0}^\infty c_j(\bfb)(\z_m^tw)^j \right)
      H_{\bfb,m}(w^m) \\
  &= w^{m-k} \left(\sum_{i=0}^d \sum_{j=j_0}^\infty a_{d-i} c_j(\bfb)w^{i+j}
             \sum_{t=0}^{m-1} \z_m^{(-k+i+j)t} \right) 
      H_{\bfb,m}(w^m) \\
  &= m w^m \left(\sum_{\substack{0\le i\le d,\; j\ge j_0\\ i+j\equiv k\pmod{m}\\}}
         a_{d-i} c_j(\bfb)w^{i+j-k}\right)
      H_{\bfb,m}(w^m).
\end{align*}
This last expression shows that the Laurent series for~$g(w)$ has the
form
\[
  g(w) \in m w^{-me} \ZZ[\bfa,\bfb,b_d^{-1}][\![w^m]\!]
  \quad\text{for some integer $e\ge0$.}
\]
But we proved earlier that $g(w)$ is a polynomial
in~$w^m\ZZ[\bfa,\bfb,w^m]$. Hence
\[
  g(w) 
  \in w^m\ZZ[\bfa,\bfb,w^m] \cap m w^{-me} \ZZ[\bfa,\bfb,b_d^{-1}][\![w^m]\!]
  = mw^m\ZZ[\bfa,\bfb,w^m],
\]
which shows that
\[
  G_{\bfa,\bfb,m,k} = \frac{1}{mw^m}g(w) \in \ZZ[\bfa,\bfb,w^m].
\]
This completes the proof of~(a).
\par
We next consider the homogeneity properties described in~(c).  We
know from~(a) that the $z$-coefficients of~$G_{\bfa,\bfb,m,k}(z)$
and~$H_{\bfb,m}(z)$ are in~$\ZZ[\bfa,\bfb]$ and~$\ZZ[\bfb]$,
respectively.  It is clear from the formula~\eqref{eqn:Gabnkwm}
for~$G_{\bfa,\bfb,m,k}$ that it is a $\QQ(\z_m)$-linear combination of
monomials of the form
\begin{align*}
  M_\bfj(\bfa,\bfb,w)  
  &= w^{-k} a_{j_1}w^{d-j_1} b_{j_2}w^{d-j_2} b_{j_3}w^{d-j_3} \cdots b_{j_m}w^{d-j_m} \\
  &= a_{j_1}b_{j_2}\cdots b_{j_m} w^{md-k-j_1-\cdots-j_m}.
\end{align*}
The~$(\bfa,\bfb)$ coefficients of these monomials are clearly
bi-homogeneous of bi-degree~$(1,m-1)$ in the variables~$(\bfa,\bfb)$.
Further, using the weights described by~\eqref{eqn:weights}, so in
particular $\deg(w)=\frac{1}{m}\deg(z)=1$, we have
\[
  \wt\bigl(M_\bfj(\bfa,\bfb,w)\bigr)
  = md-k.
\]
This completes the proof that~$G_{\bfa,\bfb,m,k}$ has the indicated
bi-degree and weight.  The proof for~$H_{\bfb,m}$ is similar, but
easier, so we leave it for the reader. 
\par
We turn to~(d).  We will make frequent use of~\ref{eqn:prodzmtneg1m1}
without further comment. The highest degree term of~$H_{\bfb,m}(w^m)$
is
\[
  \prod_{t=0}^{m-1} b_0(\z_m^tw)^d = (-1)^{(m+1)d} b_0^d w^{md},
\]
so~$H_{\bfb,m}(z)$ has the indicated form. Similarly, the highest degree term
of~$G_{\bfa,\bfb,m,k}(w^m)$ is
\begin{align*}
  \frac{1}{mw^k} \sum_{t=0}^{m-1} \z_m^{-kt} a_0 (\z_m^tw)^d 
       & \prod_{s\ne t} b_0(\z_m^sw)^d \\
  &=   \frac{a_0b_0^{m-1}w^{md-k}}{m} 
            \sum_{t=0}^{m-1} \z_m^{-kt} \prod_{s=0}^{m-1} \z_m^{ds} \\
  &= \begin{cases}
      (-1)^{(m+1)d}a_0b_0^{m-1}w^{md} &\text{if $k=0$,} \\
      0&\text{if $1\le k<m$.}\\
     \end{cases}
\end{align*}
Hence~$G_{\bfa,\bfb,m,k}(z)$ also has the indicated form.
This completes the proof of~(d).
\par
For~(e) we compute
\begin{align*}
  H_{\bfb,m}(w^m)
  &= \prod_{t=0}^{m-1} F_\bfb(\z_m^tw) \\
  &= \prod_{t=0}^{m-1} \left(b_0 \prod_{i=1}^d (\z_m^t w-\b_i)\right) \\
  &= b_0^m \prod_{i=1}^d \prod_{t=0}^{m-1} (\z_m^t w-\b_i) \\
  &= b_0^m \prod_{i=1}^d \Bigl( (-1)^m (\b_i^m-w^m) \Bigr) \\
  &= (-1)^{(m+1)d} b_0^m \prod_{i=1}^d  (w^m-\b_i^m).
\end{align*}
\par
To prove~(f), we use the fact that for any polynomials we have
\[
  \Resultant\bigl(f(w^m),g(w^m)\bigr)
  =   \Resultant\bigl(f(w),g(w)\bigr)^m.
\]
So we compute the resultant of~$G_{\bfa,\bfb,m,k}(z)$
and~$H_{\bfb,m}(z)$ with respect to the~$w$ variable and then take
the~$m^{\text{th}}$~root. We use various elementary formulas such as
\begin{align*}
  \Resultant\bigl(f(\a z),g(\a z)\bigr) 
 & = 
  \a^{(\deg f)(\deg g)}
  \Resultant\bigl(f(z),g(z)\bigr),\\
  \Resultant\bigl(Af(z),Bg(z)\bigr) 
    &= A^{\deg g}B^{\deg f}  \Resultant\bigl(f(z),g(z)\bigr),\\
  \Resultant\bigl(f(z),f'(z)\bigr) &= 
     (-1)^{(n^2-n)/2} a_0 \Disc\bigl(f(z)\bigr),
    \quad\text{where $n=\deg(f)$.}
\end{align*}
Then
\begin{align*}
  \Resultant&\bigl(G_{\bfa,\bfb,m,k}(w),H_{\bfb,m}(w)\bigr)^m \\
  &=\Resultant\bigl(G_{\bfa,\bfb,m,k}(w^m),H_{\bfb,m}(w^m)\bigr) \\
  &=\Resultant\left(
       \frac{1}{mw^k}\sum_{t=0}^{m-1} \z_m^{-kt}F_\bfa(\z_m^t w)
           \prod_{s\ne t} F_\bfb(\z_m^s w),
       \prod_{r=0}^{m-1} F_\bfb(\z_m^r w)
     \right) \\
  &=\prod_{r=0}^{m-1}\Resultant\left(
       \frac{1}{mw^k}\sum_{t=0}^{m-1} \z_m^{-kt}F_\bfa(\z_m^t w)
           \prod_{s\ne t} F_\bfb(\z_m^s w),
        F_\bfb(\z_m^r w)
     \right) \\
  &=\prod_{r=0}^{m-1}\Resultant\left(
       \frac{1}{mw^k}\z_m^{-kr}F_\bfa(\z_m^r w)
           \prod_{s\ne r} F_\bfb(\z_m^s w),
        F_\bfb(\z_m^r w)
     \right) \\
  &=\prod_{r=0}^{m-1}
           \frac{\Resultant\left(
           \z_m^{-kr}F_\bfa(\z_m^r w),
        F_\bfb(\z_m^r w)
     \right)}{\Resultant\bigl(mw^k,F_\bfb(\z_m^r w)\bigr)}
        \\*
  &\hspace{1in}{}\times
     \prod_{r=0}^{m-1} \prod_{s\ne r} 
         \Resultant\left(F_\bfb(\z_m^s w),F_\bfb(\z_m^r w)\right) \\
  &=\prod_{r=0}^{m-1}
           (\z_m^{-kr})^d (\z_m^r)^{d^2} \frac{
           \Resultant\left(F_\bfa(w),F_\bfb(w) \right)}
     {m^d b_d^k}
        \\*
  &\hspace{1in}{}\times
     \prod_{r=0}^{m-1} \prod_{s\ne r}  (\z_m^s)^{d^2}
         \Resultant\left(F_\bfb(w),F_\bfb(\z_m^{r-s} w)\right) \\
  &= \pm
       \left(\Resultant\left(F_\bfa(w),F_\bfb(w) \right)m^{-d}b_d^{-k}
          \right)^m \\
  &\hspace{1in}{}\times
    \biggl(\prod_{r=1}^{m-1}
         \Resultant\left(F_\bfb(w),F_\bfb(\z_m^{r} w)\right)\biggr)^m.
\end{align*}
Taking~$m^{\text{th}}$ roots yields 
\begin{multline}
  \label{eqn:resGabmkHabmk}
  \Resultant\bigl(G_{\bfa,\bfb,m,k}(w),H_{\bfb,m}(w)\bigr) \\*
  = \xi
       \Resultant\left(F_\bfa(w),F_\bfb(w) \right)m^{-d}b_d^{-k} 
    \prod_{r=1}^{m-1}
         \Resultant\left(F_\bfb(w),F_\bfb(\z_m^{r} w)\right)
\end{multline}
for some~$\xi\in\bfmu_{2m}$, and aside from evaluating~$\xi$, it only
remains to deal with the final product.
\par
As in~(e), we factor~$F_\bfb$ as $F_\bfb(z) = b_0\prod_{i=1}^d
(z-\b_i)$.  Then
\begin{align*}
  \prod_{r=1}^{m-1}&\Resultant\left(F_\bfb(w),F_\bfb(\z_m^{r} w)\right)\\*
  &= \prod_{r=1}^{m-1}\biggl(
         b_0^{2d}\z_m^{rd} \prod_{i,j=1}^d (\b_i-\z_m^{-r}\b_j)\biggr) \\
  &=\pm b_0^{2d(m-1)}
      \biggl( \prod_{i=1}^d \b_i^{m-1}\prod_{r=1}^{m-1}(1-\z_m^{-r})^d \biggr)
      \biggl(\prod_{i\ne j} \prod_{r=1}^{m-1} (\b_i-\z_m^{-r}\b_j)\biggr) \\
  &=\pm b_0^{(2d-1)(m-1)} b_d^{m-1} m^d
      \biggl(\prod_{i\ne j} \prod_{r=1}^{m-1} (\b_i-\z_m^{-r}\b_j)\biggr)  \\
  &=\pm b_0^{(2d-1)(m-1)} b_d^{m-1} m^d
    \left(\prod_{i\ne j} 
           \frac{\prod_{r=0}^{m-1} (\b_i-\z_m^{-r}\b_j)}{\b_i-\b_j}\right) \\
  &=\pm b_0^{(2d-1)(m-1)} b_d^{m-1} m^d
      \prod_{i\ne j} \frac{\b_i^m-\b_j^m}{\b_i-\b_j} \\
  &=\pm b_0^{m-1} b_d^{m-1} m^d
        \frac{\Disc(H_{\bfb,m})}{\Disc(F_\bfb)},
\end{align*}
where the last equality uses the formulas
\begin{align}
  \label{eqn:discFb}
  \Disc(F_\bfb) &= \pm b_0^{2d-2} \prod_{i\ne j} (\b_i-\b_j),\\
  \label{eqn:discPmFb}
  \Disc(H_{\bfb,m}) &=  \pm b_0^{m(2d-2)} \prod_{i\ne j} (\b_i^m-\b_j^m),
\end{align}
the latter of which follows from~(e). Hence
\[
  \prod_{r=1}^{m-1}\Resultant\left(F_\bfb(w),F_\bfb(\z_m^{r} w)\right)
  =\pm b_0^{m-1} b_d^{m-1} m^d
        \frac{\Disc(H_{\bfb,m})}{\Disc(F_\bfb)}.
\]
Substituting this into~\eqref{eqn:resGabmkHabmk} gives
\begin{equation}
  \label{eqn:resGHxi2}
  \Resultant(G_{\bfa,\bfb,m,k},H_{\bfb,m}) \\
  = \xi\Resultant(F_\bfa,F_\bfb)
       b_0^{m-1}b_d^{m-1-k}
        \frac{\Disc(H_{\bfb,m})}{\Disc(F_\bfb)}.
\end{equation}
\par
In order to complete the proof of~(f), it remains to evaluate~$\xi$.
Since~\eqref{eqn:resGHxi2} is an identity in~$\ZZ[\bfa,\bfb]$, we 
see that~$\xi\in\{\pm1\}$, and it suffices to compute~$\xi$ for
a single pair~$[\bfa,\bfb]\in\ZZ$ such that
$\Resultant(G_{\bfa,\bfb,m,k},H_{\bfb,m})\ne0$. 
We only sketch the proof, since for our applications, it suffices to
know that~$\xi$ is a root of unity.  Taking $F_\bfa(z)=z^d$ and
$F_\bfb(z)=(z-1)^d$, an easy calculation shows that the right-hand
side of~\eqref{eqn:resGHxi2} equals $\xi(-1)^{d(m-k)}m^{d(d-1)}$,
while a slightly more complicated calculation shows that the left-hand
side of~\eqref{eqn:resGHxi2} equals~$(-1)^dm^{d(d-1)}$.
Hence $\xi=(-1)^{d(m-k+1)}$.
\par
For~(g), we use~\eqref{eqn:discFb} and~\eqref{eqn:discPmFb} to write
\begin{equation}
  \label{eqn:DishHDiscFprodsum}
  \frac{\Disc(H_{\bfb,m})}{\Disc(F_\bfb)}
  = \pm b_0^{(m-1)(2d-2)} \prod_{i\ne j} \sum_{k=0}^{m-1} \b_i^k\b_j^{m-1-k}.
\end{equation}
The product of sums is symmetric in~$\b_1,\ldots,\b_d$, so
$\Disc(H_{\bfb,m})/\Disc(F_\bfb)$ is in $\ZZ[\bfb,b_0^{-1}]$. On the
other hand, we know that $\Disc(F_\bfb)$ is in~$\ZZ[\bfb]$ and that it
is irreducible in~$\CC[\bfb]$;
see~\cite[\S28]{vanderwaerden:algebra}. In particular, it is not
divisible by~$b_0$, so $\Disc(H_{\bfb,m})/\Disc(F_\bfb)$ is in
$\ZZ[\bfb]$. Alternatively, we can see that $\Disc(F_\bfb)$ is not
divisible by~$b_0$ in~$\ZZ[\bfb]$  directly from the formula
\[
  \Disc(b_0z^d+b_1z^{d-1}+b_d)=b_d^{d-2}(d^db_db_0^{d-1}-(-1)^d(d-1)^{d-1}b_1^d)
\]
for the discriminant of a trinomial.  This gives~(g).
\par
Finally, we see that~(h) follows immediately
from~\eqref{eqn:DishHDiscFprodsum} provided that we can show that
$\Disc(H_{\bfb,m})/\Disc(F_\bfb)$ does not vanish when~$b_0=0$.  We
will  show that~$\Disc(H_{\bfb,m})$ is not in the
ideal of~$\ZZ[\bfb]$ generated by~$b_0$. 
It is convenient to work in the ring
\[
  R = \ZZ[b_d^{-1},\g_1,\g_2,\ldots,\g_d]
  \quad\text{with $\g_i=\b_i^{-1}$.}
\]
We note that~$\g_1,\ldots,\g_d$ satisfy
\[
  z^dF_\bfb(z^{-1}) = b_0+b_1z+\cdots+b_dz^d = b_d \prod_{i=1}^d (z-\g_i),
\]
so in particular~$\g_1,\ldots,\g_d$ are algebraically independent and
integral over $\ZZ[b_d^{-1},\bfb]$. Further, 
\[
  b_0 = (-1)^d b_d\g_1\g_2\cdots\g_d,
  \quad\text{so as ideals we have}\quad
  b_0R = \g_1\g_2\cdots\g_d R.
\]
Rewriting the formula~\eqref{eqn:discPmFb} for $\Disc(H_{\bfb,m})$ in
terms of~$\g_1,\ldots,\g_d$ yields
\[
  \Disc(H_{\bfb,m}) = \pm b_d^{m(2d-2)} \prod_{i\ne j} (\g_i^m-\g_j^m).
\]
Since~$b_d\in R^*$, we are reduced to the following assertion.
\begin{equation}
  \label{eqn:prodgimgjmnotin}
  \text{Claim:}\quad
  \prod_{i\ne j} (\g_i^m-\g_j^m) \notin \g_1\g_2\cdots\g_d R.
\end{equation}
In the product~\eqref{eqn:prodgimgjmnotin} we consider the monomial
\begin{equation}
  \label{eqn:g1m2dg2m2d}
  (\g_1^m)^{2(d-1)}(\g_2^m)^{2(d-2)}(\g_3^m)^{2(d-3)}
  \cdots (\g_{d-2}^m)^{2\cdot2}(\g_{d-1}^m)^{2\cdot1}(\g_{d}^m)^{2\cdot0}.
\end{equation}
There is a \emph{unique} way to choose a term in each binomial in the
product~\eqref{eqn:prodgimgjmnotin} to get this monomial, since to get
$(\g_1^m)^{2(d-1)}$ we need to take~$\g_1^m$ in every
term~$\g_i^m-\g_j^m$ in which either~$i$ or~$j$ is~$1$; then to get
$(\g_2^m)^{2(d-2)}$ we need to take~$\g_2^m$ in every remaining
term~$\g_i^m-\g_j^m$ in which either~$i$ or~$j$ is~$2$; etc.  Hence
the product on the left-hand side of~\eqref{eqn:prodgimgjmnotin}
contains a monomial~\eqref{eqn:g1m2dg2m2d} that is not in the
ideal~$\g_1\g_2\cdots\g_d R$. (Notice that the
monomial~\eqref{eqn:g1m2dg2m2d} is not a multiple of~$\g_d$.)
This completes the proof that $\Disc(H_{\bfb,m})$ is not in the
ideal $b_0\ZZ[\bfb]$.
\end{proof}

\section{The operator $\gF_{m,k}$ as a rational map}
\label{section:Fmkasrationalmap}
In this section we show that~$\gF_{m,k}$ induces a rational map on the
projective space~$\Ratbar_d\cong\PP^{2d+1}$, and in particular, we
prove Theorem~\ref{theorem:RdmkFabmk} stated in the introduction.

It is natural to ask whether the~$\gF_{m,k}$ operators preserve the
degree of the rational map~$\f$.  The answer is clearly no.  For
example, if~$\f(z)=\sum a_iz^i\in K[z]$ is a polynomial of degree~$d$,
then Proposition~\ref{prop:gFeffectonseries} tells us that
\[
  \gF_{m,k}(\f)(z)
  =  \sum_{j=0}^{\lfloor(d-k)/m\rfloor} a_{k+jm} z^j,
\]
so
\[
  \deg\bigl(\gF_{m,k}(\f)\bigr)
  \le \left\lfloor\frac{\deg(\f)-k}{m}\right\rfloor \le\frac{1}{m}\deg(\f).
\]

We start by describing a large class of degree~$d$ rational maps
whose degree is preserved by~$\gF_{m,k}$.

\begin{definition}
Let $F(X,Y)\in K[X,Y]$ be a homogeneous polynomial of degree~$d$, and
let $m\ge1$.  We say that~$F$ is \emph{$m$-nondegenerate} if~$F(z,1)$
has degree~$d$ and if the roots~$\g_1,\ldots,\g_d$ of~$F(z,1)$ in~$\Kbar$
are nonzero and have the property that for all $i\ne j$,
\[
  \text{either $\g_i=\g_j$ or $\g_i^m\ne\g_j^m$.}
\]
If~$F$ is $m$-nondegenerate for all~$m\ge2$, we say simply that~$F$ is
\emph{nondegenerate}.
\end{definition}

\begin{definition}
We set
\[
  \Rat_d^\mnondeg
  = \{ \f_{\bfa,\bfb}\in\Rat_d : \text{$F_\bfb$ is $m$-nondegenerate} \}.
\]
\end{definition}

\begin{corollary} 
\label{theorem:RmkdonRatd}
Let $m\ge2$ and $0\le k<m$. 
\begin{parts}
\Part{(a)}
Identifying $\Rat_d$ as a subset of~$\PP^{2d+1}$, the
set~$\Rat_d^\mnondeg$ is the complement of a hypersurface
of~$\PP^{2d+1}$.
\Part{(b)}
All $\f_{\bfa,\bfb}\in\Rat_d^\mnondeg$ satisfy
\[
  \deg_z\bigl( \gF_{m,k}(\f_{\bfa,\bfb}) \bigr) = d.
\]
\end{parts}
\end{corollary}
\begin{proof}
(a) Indeed, we see from Proposition~\ref{prop:resultantGH}(e,g,h)
that~$F(z)=F_\bfb(z)$ is $m$-nondegenerate if and only if $b_0b_d\ne0$
and~$\Disc(H_{\bfb,m})/\Disc(F_\bfb)\ne0$.  But
Proposition~\ref{prop:resultantGH}(g) tells us
that~$\Disc(H_{\bfb,m})/\Disc(F_\bfb)\in\ZZ[\bfb]$,
so~$\Rat_d^\mnondeg$ is the complement of the hypersurface
defined by 
\[
  b_0b_d\Disc(H_{\bfb,m})/\Disc(F_\bfb)=0.
\]
\par\noindent(b)\enspace
This is immediate from Proposition~\ref{prop:resultantGH}(f,h) and the
definition of $m$-nondegeneratcy. We note that the nondegeneracy
includes the condition that~$b_0b_d\ne0$, which is needed due to the
$b_0^{m-1}b_d^{m-1-k}$ factor in~\eqref{eqn:resGHxi}, although for
$k=m-1$, there are maps with $b_d=0$ satisfying
$\deg_z\bigl(\gF_{m,k}(\f_{\bfa,\bfb}) \bigr) = d$.
\end{proof}

\subsection{Proof of Theorem~\textup{\ref{theorem:RdmkFabmk}(a,b)}}
Proposition~\ref{prop:resultantGH} tells us that
$G_{\bfa,\bfb,m,k}(z)$ and $H_{\bfb,m}(z)$ are of degree at most~$d$
in~$z$, and that their $z$-coefficients are homogeneous polynomials of
degree~$m$ in~$\ZZ[\bfa,\bfb]$. We write
\[
  G_{\bfa,\bfb,m,k}(z) = \sum_{i=0}^d G_{m,k,i}(\bfa,\bfb)z^{d-i}
  \quad\text{and}\quad
  H_{\bfb,m}(z) = \sum_{i=0}^d H_{m,i}(\bfb)z^{d-i}.
\]
More precisely, Proposition~\ref{prop:resultantGH}(d) tells us that
\begin{align*}
  H_{m,0} &= (-1)^{d(m+1)} b_0^m,\\
  G_{m,0,0} &= (-1)^{d(m+1)} a_0b_0^{m-1},\\
  G_{m,k,0} &= 0\quad\text{for $1\le k<m$.}
\end{align*}
Further, the formula
$\gF_{m,k}(\f_{\bfa,\bfb})(z)=G_{\bfa,\bfb,m,k}(z)/H_{\bfb,m}(z)$ in
Proposition~\ref{prop:resultantGH}(b) implies that the rational 
map~$\gR_{d,m,k}:\PP^{2d+1}\dashrightarrow\PP^{2d+1}$ is given by
\begin{equation}
  \label{eqn:FmkeqGmk0Hmd}
  \gR_{d,m,k} = [G_{m,k,0},\ldots,G_{m,k,d},H_{m,0},\ldots,H_{m,d}].
\end{equation}
Hence~$\gR_{d,m,k}$ is a rational map of degree at most~$m$.  We are
next going to show that
\[
  G_{m,k,0},\ldots,G_{m,k,d},H_{m,0},\ldots,H_{m,d}
\]
have no nontrivial common factor in the polynomial
ring~$\ZZ[\bfa,\bfb]$, which will complete the proof that $\gR_{d,m,k}$
has degree exactly equal to~$m$.
\par
Let
\[
  W = \left\{[\bfa,\bfb]\in\PP^{2d+1} : 
    \begin{array}{c}
      G_{m,k,0}(\bfa,\bfb) = \cdots = G_{m,k,d}(\bfa,\bfb) = 0\\
      H_{m,0}(\bfb) = \cdots = H_{m,d}(\bfb) = 0\\
    \end{array}
  \right\},
\]
so~\eqref{eqn:FmkeqGmk0Hmd} tells us that~$Z(\gR_{d,m,k})\subset W$.  We
first observe that if~$\bfb=0$, then the formulas
for~$G_{\bfa,\bfb,m,k}(z)$ and~$H_{\bfb,m}(z)$ given in
Proposition~\ref{prop:resultantGH} imply that ~$G_{\bfa,\bfb,m,k}(z)$
and~$H_{\bfb,m}(z)$ are both identically~$0$, which shows
that~$\{\bfb=\bfzero\}\subset W$.  Conversely, let $[\bfa,\bfb]\in
W$. In particular, every coefficient of~$H_{\bfb,m}(z)$ vanishes.
Since $H_{\bfb,m}(w^m)=\prod F_\bfb(\z_m^tw)$ by definition, it
follows that there is some~$t$ such that~$F_\bfb(\z_m^tw)=0$ as a
polynomial in~$w$, which in turn implies that~$\bfb=\bfzero$.  This
gives the other inclusion, so we have proven that
\[
  W = \bigl\{[\bfa,\bfb]\in\PP^{2d+1} : \bfb=\bfzero\bigr\}.
\]
\par
Suppose now that $\deg\gR_{d,m,k}<m$. As noted earlier, this implies
that $G_{m,k,0},\ldots,G_{m,k,d},H_{m,0},\ldots,H_{m,d}$ have a
nontrivial common factor~$u(\bfa,\bfb)\in \ZZ[\bfa,\bfb]$. But
then~$W$  contains the subvariety \text{$\{u=0\}$}, which has
dimension~$2d$, contradicting the fact that 
\[
  \dim W=\dim\{\bfb=\bfzero\}=d. 
\]
This shows that~$\deg\gR_{d,m,k}=m$, which completes the proof of
Theorem~\ref{theorem:RdmkFabmk}(a), while simultaneously proving that
\[
  \Zcal(\gR_{d,m,k}) = W = \bigl\{[\bfa,\bfb]\in\PP^{2d+1} : \bfb=\bfzero\bigr\}.
\]
Finally, let $[\bfa,\bfb]\in\PP^{2d+1}\setminus \Zcal(\gR_{d,m,k})$. Then
\begin{align}
  \label{eqn:ZRdmkcomp}
  \gR_{d,m,k}&(\bfa,\bfb)\in\Zcal(\gR_{d,m,k}) \notag\\
  &\quad\Longleftrightarrow\quad
  H_{\bfb,m}(z) = 0~\text{as a $z$-polynomial,} \notag\\
  &\quad\Longleftrightarrow\quad
  H_{\bfb,m}(w^m)
     =\prod_{t=0}^{m-1} F_\bfb(\z_m^tw)=0~\text{as a $w$-polynomial,} \notag\\
  &\quad\Longleftrightarrow\quad
  F_\bfb(\z_m^tw)=0~\text{as a $w$-polynomial, for some $t$,} \notag\\
  &\quad\Longleftrightarrow\quad
  \bfb=\bfzero \notag\\
  &\quad\Longleftrightarrow\quad
  [\bfa,\bfb]\in\Zcal(\gR_{d,m,k}).
\end{align}
This completes the proof of Theorem~\ref{theorem:RdmkFabmk}(b).

\subsection{Computation of a Jacobian matrix}
The proof of the remaining parts of Theorem~\ref{theorem:RdmkFabmk}
is more complicated and requires some preliminary results.

The rational map~$\gR_{d,m,k}(\bfa,\bfb)$ is given by a
list of~$2d+2$ homogeneous polynomials of degree~$m$. We write
\[
  \Jcal_{d,m,k}(\bfa,\bfb) = \operatorname{Jac} \gR_{d,m,k}(\bfa,\bfb)
\]
for the associated Jacobian matrix. We note that
since~$G_{\bfa,\bfb,m,k}$ has degree~$1$ in~$\bfa$ and~$H_{\bfb,m}$ is
independent of~$\bfa$ and~$k$, the matrix~$\Jcal_{d,m,k}$ has block form
\begin{equation}
  \label{eqn:JabA0CD}
  \Jcal_{d,m,k}(\bfa,\bfb)
  = \begin{pmatrix}
    A_{d,m,k}(\bfb) & \bfzero \\
    C_{d,m,k}(\bfa,\bfb) & D_{d,m}(\bfb) \\
  \end{pmatrix},
\end{equation}
In particular, the Jacobian determinant
\begin{equation}
  \label{eqn:detJdetAdetD}
  \det\Jcal_{d,m,k} = (\det A_{d,m,k}) (\det D_{d,m}) \in \ZZ[\bfb]
\end{equation}
is independent of~$\bfa$. 

\begin{lemma}
\label{lemma:Ddm0mAdm0}
With notation as in~\eqref{eqn:JabA0CD}, in the case that $k=0$ we
have
\[
   D_{d,m} = m A_{d,m,0}.
\]
\end{lemma}
\begin{proof}
With our usual identification of~$z=w^m$, we have by definition
that the $(i,j)$'th entry of~$A_{d,m,0}$ is the coefficient of~$z^{d-j}$
in the partial derivative
\begin{align*}
  \frac{\partial G_{\bfa,\bfb,d,m,0}(z)}{\partial a_i}
  &= \frac{\partial\hfill}{\partial a_i} \left(
      \frac{1}{m}\sum_{t=0}^{m-1} F_\bfa(\z_m^t w)
           \prod_{s\ne t} F_\bfb(\z_m^s w) \right) \\
  &= \frac{1}{m} \sum_{t=0}^{m-1} (\z_m^tw)^{d-i} \prod_{s\ne t} F_\bfb(\z_m^sw).
\end{align*}
Similarly, the $(i,j)$'th entry of~$D_{d,m}$ is the coefficient
of~$z^{d-j}$ in the partial derivative
\begin{align*}
  \frac{\partial H_{\bfb,d,m}(z)}{\partial b_i} 
  &= \frac{\partial\hfill}{\partial b_i} \left(
             \prod_{t=0}^{m-1} F_\bfb(\z_m^t w) \right) \\
  &=  \sum_{t=0}^{m-1} 
      (\z_m^tw)^{d-i} \prod_{s\ne t} F_\bfb(\z_m^sw).
\end{align*}
Comparing these formulas shows that $D_{d,m} = m A_{d,m,0}$.
\end{proof}

\begin{example}
We illustrate Lemma~\ref{lemma:Ddm0mAdm0} and the block
form~\eqref{eqn:JabA0CD} of~$\Jcal_{d,m,k}$ by computing
\[
  \Jcal_{2,2,0}
  = \left(\begin{array}{c|c}
     \begin{matrix}
        b_0&b_2&0\\
        0&-b_1&0\\
        0&b_0&b_2\\
     \end{matrix}
     &
     \begin{matrix}
        \hspace{10pt}0_{\phantom0}\hspace{10pt}&0&
        \hspace{10pt}0_{\phantom0}\hspace{10pt}\\
        0_{\phantom0}&0&0\\
        0_{\phantom0}&0&0\\
     \end{matrix}
     \\ \hline
     \begin{matrix}
        a_0&a_2&0\\
        0&-a_1&0\\
        0&a_0&a_2\\
     \end{matrix}
     &
     \begin{matrix}
        2b_0&2b_2&0\\
        0&-2b_1&0\\
        0&2b_0&2b_2\\
     \end{matrix}
     \\
  \end{array}\right)
\]
\end{example}

The next lemma, which includes a somewhat complicated
calculation, is the key to showing that the matrix~$A_{d,m,0}$ is
generically non-singular in all characteristics.

\begin{lemma}
\label{lemma:degwtadmkij}
Write the entries of the matrix~$A_{d,m,k}(\bfb)$ as
\[
  A_{d,m,k}(\bfb) = \bigl( \a_{d,m,k}(\bfb)_{i,j} \bigr) _{0\le i,j\le d}.
\]
\begin{parts}
\Part{(a)}
Then $\a_{d,m,k}(\bfb)_{i,j}\in\ZZ[\bfb]$ is homogeneous for both
degree and weight, and satisfies
\begin{align*}
  \deg \a_{d,m,k}(\bfb)_{i,j} &= m-1, \\
  \wt \a_{d,m,k}(\bfb)_{i,j} &= mj - i - k.
\end{align*}
\Part{(b)}
$\det A_{d,m,k}(\bfb)\in\ZZ[\bfb]$ is homogeneous for both degree
and weight, and satisfies
\begin{align*}
  \deg\det A_{d,m,k} &= (m-1)(d+1),\\
  \wt\det A_{d,m,k} &= \frac12 (m-1)(d^2+d) - k(d+1).
\end{align*}
\Part{(c)}
Let $I_j\subset\ZZ[\bfb]$ be the ideal generated by~$b_0,b_1,\ldots,b_{j-1}$, 
where by convention we set~$I_0=(0)$.
Then
\[
   \a_{d,m,0}(\bfb)_{i,j} \equiv 0 \pmod{I_j}
  \quad\text{for all $i>j$,}
\]
and for~$i=j$, we have
\[
   \a_{d,m,0}(\bfb)_{j,j} \equiv (-1)^{(m+1)(d-j)} b_j^{m-1} \pmod{I_j}
\]
\Part{(d)}
When $\det A_{d,m,0}(\bfb)$ is written as a polynomial in~$\ZZ[\bfb]$,
it includes the monomial
\[
  (-1)^{(m+1)(d^2+d)/2}(b_0b_1\cdots b_d)^{m-1}.
\]
\end{parts}
\textup(Note that \textup{(c)} and \textup{(d)} refer to the case
that $k=0$. For $1\le k<m$, see Lemma~$\ref{lemma:Aprimekge0}$.\textup)
\end{lemma}
\begin{proof}
The $(i,j)$'th entry of $A_{d,m,k}$ is equal to the coefficient
of~$z^{d-j}$ in the partial derivative
\begin{equation}
  \label{eqn:dGabdmkzai}
  \frac{\partial G_{\bfa,\bfb,d,m,k}(z)}{\partial a_i}
  = \frac{1}{mw^k} \sum_{t=0}^{m-1} 
     \z_m^{-kt} (\z_m^tw)^{d-i} \prod_{s\ne t} F_\bfb(\z_m^sw).
\end{equation}
A typical monomial in the right-hand side of~\eqref{eqn:dGabdmkzai} is
a $\ZZ[\z_m]$-multiple of
\[
  w^{-k}w^{d-i} b_{u_1}w^{d-u_1} \cdots b_{u_{m-1}}w^{d-u_{m-1}}
  = b_{u_1}\cdots b_{u_{m-1}} w^{md-i-k - u_1-\cdots-u_{m-1}}.
\]
This quantity will be a multiple of~$z^{d-j}=w^{m(d-j)}$ if
and only if
\[
  md - i - k - (u_1+\cdots+u_{m-1}) = m(d-j).
\]
Hence $\a_{d,m,k}(\bfb)_{i,j}$ is a sum of monomials
whose  $\bfb$-degree is~$m-1$ and whose~$\bfb$-weight is
\[
  \wt(b_{u_1}\cdots b_{u_{m-1}}) 
  = u_1+\cdots+u_{m-1}
  = mj - i - k.
\]
This completes the proof of~(a).
\par
For~(b), we see that each monomial in~$\det A_{d,m,k}$ is a product
of~$d+1$ homogeneous polynomials of degree~$m-1$, which shows
that~$\det A_{d,m,k}$ is homogeneous of degree~$(m-1)(d+1)$.
Further, if~$\pi\in\Scal_{d+1}$ is any permuatation, then~$\det
A_{d,m,k}$ is a linear combination of terms having weight
\begin{align*}
  \wt \left( \prod_{i=0}^d \a_{d,m,k}(\bfb)_{i,\pi(i)} \right)
  &= \sum_{i=0}^d \wt \a_{d,m,k}(\bfb)_{i,\pi(i)}  \\
  &= \sum_{i=0}^d \bigl( m\pi(i) - i - k  \bigr) \\
  &=  (m-1) d(d+1)/2 - k(d+1).
\end{align*}
This completes the proof of~(b).
\par
The weight and degree formulas from~(a) say
that~$\a_{d,m,0}(\bfb)_{i,j}$ is a linear combination of terms of the
form
\[
  b_0^{e_0}b_1^{e_1}\cdots b_d^{e_d}
  \quad\text{with}\quad
  \sum_{t=0}^d e_t = m-1 \quad\text{and}\quad \sum_{t=0}^d te_t = mj-i.
\]
Suppose that 
\[
  \a_{d,m,0}(\bfb)_{i,j} \not\equiv 0 \pmod{I_j}.
\]
This means that~$\a_{d,m,0}(\bfb)_{i,j}$ includes a monomial
having~$e_0=\cdots=e_{j-1}=0$, i.e., a monomial of the form
\[
  b_j^{e_j}\cdots b_d^{e_d}
  \quad\text{with}\quad
  \sum_{t=j}^d e_t = m-1 \quad\text{and}\quad \sum_{t=j}^d te_t = mj-i.
\]
This leads to the inequality
\begin{equation}
  \label{eqn:mjigejm1}
  mj-i = \sum_{t=j}^d te_t
  \ge \sum_{t=j}^d je_t
  = j(m-1),
\end{equation}
which implies that $i\le j$. We have thus shown that
\begin{equation}
  \label{eqn:adm0bijne0iffilej}
  \a_{d,m,0}(\bfb)_{i,j} \not\equiv 0 \pmod{I_j}
  \quad\Longrightarrow\quad
  i \le j.
\end{equation}
The contrapositive of~\eqref{eqn:adm0bijne0iffilej} is the first part
of~(c).
\par
For the second part, we suppose that
\[
  \a_{d,m,0}(\bfb)_{i,j} \not\equiv 0 \pmod{I_j}
  \quad\text{and}\quad
  i=j.
\]
Then~$mj-i=j(m-1)$, so the middle inequality in~\eqref{eqn:mjigejm1}
is an equality. Hence
\[
  0 = \sum_{t=j}^d te_t - \sum_{t=j}^d je_t = \sum_{t=j}^d (t-j)e_t .
\]
Every term $(t-j)e_t$ is non-negative, so we must have $(t-j)e_t=0$
for all $j\le t\le d$, which implies that $e_t=0$ for all
$j+1\le t\le d$. It follows that the only monomial appearing
in~$\a_{d,m,0}(\bfb)_{i,j}$ that is not in~$I_j$ has the
form~$b_j^{e_j}$, and by degree considerations we must
have~$e_j=m-1$. This proves that
\[
  \a_{d,m,0}(\bfb)_{j,j} \equiv \g b_j^{m-1} \pmod{I_j}
  \quad\text{for some constant $\g\in\ZZ[\z_m]$.}
\]
In order to complete the proof of the second part of~(c), it remains
to compute the constant~$\g$.
\par
We are looking for the coefficient of $b_j^{m-1}w^{m(d-j)}$ in the
expression (note that $k=0$ by assumption)
\[
  \frac{1}{m} \sum_{t=0}^{m-1} (\z_m^tw)^{d-j} \prod_{s\ne t} F_\bfb(\z_m^sw).
\]
The only way to get~$b_j^{m-1}$ is to take the $b_j(\z_m^sw)^{d-j}$
term in each~$F_\bfb(\z_m^sw)$ appearing in the product. This gives
\begin{align*}
  \frac{1}{m} \sum_{t=0}^{m-1} &(\z_m^tw)^{d-j} \prod_{s\ne t} b_j(\z_m^sw)^{d-j} \\*
  &= \frac{1}{m} b_j^{m-1} w^{m(d-j)}
         \sum_{t=0}^{m-1} \z_m^{(d-j)t} \prod_{s\ne t} \z_m^{(d-j)s} \\
  &= \frac{1}{m} b_j^{m-1} w^{m(d-j)} \sum_{t=0}^{m-1} (-1)^{(m+1)(d-j)}
       \quad\text{from~\ref{eqn:prodzmtneg1m1},} \\*
  &= (-1)^{(m+1)(d-j)} b_j^{m-1} w^{m(d-j)}.
\end{align*}
This proves that~$\g=(-1)^{(m+1)(d-j)}$, which completes the proof of~(c).
\par
Using~(c), we see that the entries of the matrix~$A_{d,m,k}(\bfb)$
have the form
\[
  \begin{pmatrix}
    \pm b_0^{m-1} & * & * & * & \cdots & * \\
    I_0 & \pm b_1^{m-1}+I_1 & * & * & \cdots & * \\
    I_0 & I_1 & \pm b_2^{m-1}+I_2  & * & \cdots & * \\
    I_0 & I_1 & I_2 & \pm b_3^{m-1}+I_3 & \cdots & * \\
    \vdots &\vdots &\vdots &\vdots & \ddots & \vdots \\
    I_0 & I_1 & I_2 & I_3 & \cdots & \pm b_d^{m-1}+I_d \\
  \end{pmatrix},
\]
where we write~$I_j$ to indicate an element of the ideal~$I_j$, and
where stars indicate arbitrary elements of~$\ZZ[\bfb]$.
\par
We now consider how we might obtain the monomial~$(b_0b_1\ldots
b_d)^{m-1}$ in the expansion of~$\det
A_{d,m,k}(\bfb)$. Since~$I_0=(0)$, the only nonzero entry in the first
column is the top entry of~$\pm b_0^{m-1}$, so we expand on that
element and delete the first row and column. But we've now used up all
of our allowable factors of~$b_0$, so when we take the determinant of
the remaining $d\times d$ submatrix, we're not allowed to use any
monomials containing a~$b_0$. Equivalently, we might as well set
$b_0=0$ before taking the determinant of the $d\times d$
submatrix. When we do this, since $I_1=(b_0)$, the only nonzero entry
in the first column (of the submatrix) is the top entry, which is~$\pm
b_1^{m-1}$. Expanding on this entry and deleting the top row and
column, we've now accumulated a factor of~$(b_0b_1)^{m-1}$, which uses
up all of the allowable factors of~$b_0$ and~$b_1$. This means that we
can set $b_0=b_1=0$ before taking the determinant of the remaining
$(d-1)\times (d-1)$ submatrix. Since $I_2=(b_0,b_1)$, the first column
of the $(d-1)\times (d-1)$ submatrix is zero except for the top entry
of~$\pm b_2^{m-1}$. Continuing in this fashion, we see that the
monomial~$(b_0b_1\ldots b_d)^{m-1}$ appears in the expansion of~$\det
A_{d,m,k}(\bfb)$ with coefficient~$\pm1$.
\par
This would suffice for most purposes, but we can use the explicit
formula for the sign in~(c) to exactly determine the coefficient.
Thus $\det A_{d,m,k}(\bfb)$ contains the
monomial~$(-1)^\mu(b_0b_1\ldots b_d)^{m-1}$ with
\[
  \mu \equiv
  \sum_{j=0}^d (m+1)(d-j) = (m+1)\frac{d^2+d}{2}  \pmod{2}.
\]
This completes the proof of Lemma~\ref{lemma:degwtadmkij}.
\end{proof}

\subsection{Proof of Theorem~\textup{\ref{theorem:RdmkFabmk}(c)}}
Our goal is to prove that the rational map
$\gR_{d,m,0}:\PP^{2d+1}_\ZZ\dashrightarrow\PP^{2d+1}_\ZZ$ is dominant.
We observe that Lemma~\ref{lemma:degwtadmkij}(d) and the elementary
Jacobian formula~\eqref{eqn:detJdetAdetD} imply that the Jacobian
determinant of the rational map~$\gR_{d,m,0}$, considered as a
homogeneous polynomial in~$\ZZ[\bfb]$, includes a monomial of the form
\[
  m^{d+1} (b_0b_1\ldots b_d)^{2m-2}.
\]
It follows that
\[
  \det\Jcal_{d,m,k}(\bfa,\bfb) \ne 0
  \quad\text{in $\ZZ[m^{-1}][\bfb]$.}
\]
This proves that~$\gR_{d,m,0}$ is a dominant rational map provided
that~$m$ is invertible, i.e., as long as we're not working in
characteristic~$p$ for some prime~$p$ dividing~$m$. In particular, it
proves that~$\gR_{d,m,0}$ is a dominant rational self-map
of~$\PP^{2d+1}_\QQ$.
\par
However, for characteristics~$p$ dividing~$m$, this tangent space
argument will not work. Indeed,  we will soon see that the
map~$\gR_{d,m,0}$ is inseparable over~$\FF_p$.  So we proceed as
follows.
Theorem~\ref{theorem:RdmkFabmk}(b) says that the indeterminacy locus
of~$\gR_{d,m,k}$ is the set
\[
  \Zcal = \Zcal(\gR_{d,m,k}) 
   = \bigl\{ [\bfa,\bfb]\in\PP^{2d+1} : \bfb=\bfzero \bigr\},
\]
and that~$\gR_{d,m,k}$ induces a \emph{morphism}
\[
  \gR_{d,m,k} : \PP^{2d+1}_\ZZ \setminus \Zcal
   \longrightarrow \PP^{2d+1}_\ZZ \setminus \Zcal.
\]
We note that~$\Zcal$ is independent of~$m$ and~$k$, so
compositions of various~$\gR_{d,m,k}$ for a fixed~$d$ and
different~$m$ and~$k$ give well-defined rational self-maps
of~$\PP^{2d+1}_\ZZ$, since they are self-morphisms of the Zariski
dense subset $\PP^{2d+1}_\ZZ \setminus \Zcal$.  Contained within this
set is the Zariski dense set on which~$\gR_{d,m,k}$ agrees with the
Landen transform~$\gF_{m,k}$
(cf.\ Corollary~\ref{theorem:RmkdonRatd}(a)), so the composition formula
in Proposition~\ref{prop:compositionlaw} implies the analogous formula
\begin{equation}
  \label{eqn:RdmkRdnl}
  \gR_{d,m,k} \circ \gR_{d,n,\ell} = \gR_{d,mn,kn+\ell}.
\end{equation}
The composition formula~\eqref{eqn:RdmkRdnl} is valid as rational
self-maps of~$\PP^{2d+1}_\ZZ$. Taking~$k=\ell=0$
in~\eqref{eqn:RdmkRdnl} and applying it repeatedly, we see that if~$m$
has a factorization $m=p_1p_2\cdots p_r$ as a product of (not
necessarily distinct) primes, then
\begin{equation}
  \label{eqn:RdmprodRdpi}
  \gR_{d,m,0} = \gR_{d,p_1,0} \circ \gR_{d,p_2,0} \circ \cdots \circ \gR_{d,p_r,0}.
\end{equation}
Hence in order to prove that~$\gR_{d,m,0}$ is a dominant rational self-map
of~$\PP^{2d+1}_\ZZ$, it suffices to consider the case that~$m=p$ is prime.
Further, since we have already proven that~$\gR_{d,p,0}$ 
is dominant over~$\ZZ[p^{-1}]$, it suffices to prove that the reduction
modulo~$p$,
\[
  \tilde\gR_{d,p,0} : \PP^{2d+1}_{\FF_p} \dashrightarrow \PP^{2d+1}_{\FF_p},
\]
is dominant.
\par

In order to analyze~$\tilde\gR_{d,p,0}$, it is convenient to
write~$\FF_p$ as the quotient
field
\[
  \FF_p = \frac{\ZZ[\z_p]}{\gp}
  \quad\text{with $\gp$ the ideal $\gp=(1-\z_p)\ZZ[\z_p]$.}
\]
Then using the fact that~$\z_p\equiv1\pmod{\gp}$, we see that
\[
  H_{\bfb,p}(w^p)
  = \prod_{t=0}^{p-1} F_\bfb(\z_p^tw)
  \equiv F_\bfb(w)^p
  \equiv F_{\bfb^p}(w^p)
  \pmod\gp,
\]
where we write~$\bfb^p$ for the $p$-power Frobenius map applied to the
coordinates of~$\bfb$. This proves that the last~$d+1$ coordinate
functions of~$\tilde\gR_{d,p,0}(\bfa,\bfb)$ are~$b_0^p,\ldots,b_d^p$,
i.e.,~$\tilde\gR_{d,p,0}$ has the form
\[
  \tilde\gR_{d,p,0}(\bfa,\bfb) = [ \tilde G_{\bfa,\bfb,p,0},
        b_0^p, b_1^p, \ldots, b_d^p],
\]
where we write~$\tilde G_{\bfa,\bfb,p,0}$ for the list of
$z$-coefficients of~$G_{\bfa,\bfb,p,0}(z)$, reduced modulo~$\gp$.
\par
Tracking through the various definitions and using the fact that the
polynomial~$G_{\bfa,\bfb,p,k}(z)$ is homogenerous of degree~$1$
in~$\bfa$, we see that~$G_{\bfa,\bfb,p,k}(z)$ and the
matrix~$A_{d,p,k}(\bfb)=\bigl(\a_{d,p,k}(\bfb)\bigr)$ are related by
the formula
\begin{equation}
  \label{eqn:Gabmkzsumaiadmkbij}
  G_{\bfa,\bfb,p,k}(z) = \sum_{j=0}^d \left(
     \sum_{i=0}^d a_i \a_{d,p,k}(\bfb)_{i,j} \right) z^{d-j}.
\end{equation}
\par
Let
\[
  K = \overline{\FF_p(u_0,\ldots,u_d,v_0,\ldots,v_d)},
\]
where~$u_0,\ldots,u_d,v_0,\ldots,v_d$ are algebraically independent
over~$\FF_p$, and the overline denotes an algebraic closure of the
indicated rational function field. We are going to show that 
\[
  [u_0,\ldots,u_d,v_0,\ldots,v_d] 
   \in \tilde\gR_{d,p,0}\bigl(\PP^{2d+1}(K)\setminus\Zcal\bigr).
\]
This will show
that~$\tilde\gR_{d,p,0}:\PP^{2d+1}_{\FF_p}\dashrightarrow\PP^{2d+1}_{\FF_p}$
is generically surjective, and hence that~$\tilde\gR_{d,p,0}$ is
dominant.  \par To show that~$[u_0,\ldots,v_d]$ is in the image
of~$\tilde\gR_{d,p,0}$, we first let
\[
  w_i=v_i^{1/p}\in K\quad\text{for $0\le i\le d$,}
\]
so in particular,~$w_0,\ldots,w_d$ are also algebraically independent
over~$\FF_p$. It follows from Lemma~\ref{lemma:degwtadmkij}(d) that
the matrix~$A_{d,p,0}(\bfw)$ is non-singular, i.e., $\det
A_{d,p,0}(\bfw)\ne0$ in~$K$.  (Note how we use here the fact
that $\det A_{d,p,0}(\bfb)$ includes a monomial whose coefficient
is~$\pm1$, and hence a monomial that persists when we reduce
modulo~$p$.)  This allows us to define a vector
\[
  \bfx = A_{d,p,0}(\bfw)^{-1} \bfu \in K^{d+1}.
\]
Using~\eqref{eqn:Gabmkzsumaiadmkbij}, which says that the
coefficients of~$G_{\bfa,\bfb,p,k}(z)$ are linear functions of~$\bfa$
whose coefficients form  the matrix~$A_{d,p,k}(\bfb)$, we see that
\[
  \tilde\gR_{d,p,0}(\bfx,\bfw) = [\bfu,\bfv].
\]
Hence $\tilde\gR_{d,p,0}$ is dominant, which completes the proof
of Theorem~\ref{theorem:RdmkFabmk}(c).

\begin{remark}
Theorem~\ref{theorem:RdmkFabmk} describes the algebraic degree of the
rational map~$\gR_{d,m,k}$, where in general, the algebraic degree of
a rational map~$\f:\PP^N\dashrightarrow\PP^N$ of projective space is the
integer~$d$ satisfying $\f^*\Ocal_{\PP^N}(1)=\Ocal_{\PP^N}(d)$. It is
also of interest to compute the separable and inseparable degrees of 
dominant rational maps $\f:X\dashrightarrow Y$ of equidimensional
varieties. By definition, these are the separable and inseparable
degrees of the associated extension $K(X)/\f^*K(Y)$ of function
fields. (Over~$\CC$, the separable degree is equal to the topological
degree of~$\f$, i.e.,~$\#\f^{-1}(y)$ for a generic point~$y\in
Y(\CC)$.)  The proof of Theorem~\ref{theorem:RdmkFabmk}(c) shows that
in characteristic~$0$, the induced map
$\gR_{d,m,0}:\PP^{2d+1}_\QQ\dashrightarrow\PP^{2d+1}_\QQ$ has (separable)
degree~$m^d$, while in characteristic~$p$,  we factor $m=p^en$ with
$p\nmid n$, and then $\gR_{d,m,0}:\PP^{2d+1}_{\FF_p}\dashrightarrow\PP^{2d+1}_{\FF_p}$
has separable degree~$n^d$ and inseparable degree~$p^{ed}$.  The same
statements are true for~$\gR_{d,m,k}$ as self-maps
of~$\{a_0=0\}\cong\PP^{2d}$.
\end{remark}

\subsection{Proof of Theorem~\textup{\ref{theorem:RdmkFabmk}(d)}}
The proof of~(d) is similar to~(c), but longer and computationally
more complicated, so we only give an outline and leave the details
to the reader.
\par
We use a prime to denote restriction to the hyperplane~$\{a_0=0\}$
in~$\PP^{2d+1}$. So for example~$\gR_{d,m,k}'$ is the restriction
of~$\gR_{d,m,k}$ to~$\{a_0=0\}$, and~$\Jcal_{d,m,k}'(\bfa,\bfb)$ is
the Jacobian matrix of~$\gR_{d,m,k}'$. We observe
that~$\Jcal_{d,m,k}'(\bfa,\bfb)$ is obtained
from~$\Jcal_{d,m,k}(\bfa,\bfb)$ by deleting the first column and the
first row of~$\Jcal_{d,m,k}(\bfa,\bfb)$, and then setting~$a_0=0$.
Looking at the block form~\eqref{eqn:JabA0CD}
of~~$\Jcal_{d,m,k}(\bfa,\bfb)$, we see
that~$\Jcal_{d,m,k}'(\bfa,\bfb)$ has the form
\begin{equation}
  \label{eqn:JabA0CDprime}
  \Jcal_{d,m,k}'(\bfa,\bfb)
  = \begin{pmatrix}
    A_{d,m,k}'(\bfb) & \bfzero \\
    C_{d,m,k}'(\bfa,\bfb) & D_{d,m}(\bfb) \\
  \end{pmatrix},
\end{equation}
where~$A_{d,m,k}'$ is the~$d$-by-$d$ matrix obtained by deleting the
first column and row of~$A_{d,m,k}$, and~$D_{d,m}$ is the
\text{$(d+1)$}-by-\text{$(d+1)$} matrix already appearing
in~$\Jcal_{d,m,k}$.
\par
We note that Lemmas~\ref{lemma:Ddm0mAdm0}
and~\ref{lemma:degwtadmkij}(d) tell us that $\det D_{d,m}(\bfb)$
includes the monomial
\begin{equation}
  \label{eqn:monominDdm}
  m^{d+1} (-1)^{(m+1)(d^2+d)/2}(b_0b_1\cdots b_d)^{m-1}.
\end{equation}
\par
We recall Lemma~\ref{lemma:degwtadmkij}(c,d) gives various formulas
when~$k=0$. The next result gives analogous formulas for~$1\le k<m$. 

\begin{lemma}
\label{lemma:Aprimekge0}
Let $1\le k<m$, and write the entries of the matrix~$A_{d,m,k}'$ as
$A_{d,m,k}'(\bfb)=\bigl(\a_{d,m,k}'(\bfb)_{i,j}\bigr)$.
\begin{parts}
\Part{(a)}
For $1\le j\le d$, we write~$I_j'\subset\ZZ[\bfb]$ for the ideal
generated by
\[
  b_0,b_1,\ldots,b_{j-2},b_{j-1}^{k+1}.
\]
Then
\begin{align*}
  \a_{d,m,k}'(\bfb)_{i,j} &\equiv 0 \pmod{I_j'}
    \quad\text{for all $i>j$,} \\
  \a_{d,m,k}'(\bfb)_{j,j} &\equiv  (-1)^{(m+1)(d-j)+k}b_{j-1}^kb_j^{m-1-k}
          \pmod{I_j'}.
\end{align*}
\Part{(b)}
When $\det A_{d,m,k}'(\bfb)$ is written as a polynomial in~$\ZZ[\bfb]$,
it includes the monomial
\[
  (-1)^{(m+1)(d^2-d)/2+dk} b_0^k(b_1b_2\cdots b_{d-1})^{m-1}b_d^{m-1-k}.
\]
\end{parts}
\end{lemma}
\begin{proof}
We omit the proof of Lemma~\ref{lemma:Aprimekge0}, which is similar to
the proof of Lemma~\ref{lemma:degwtadmkij}(c,d).
\end{proof}

Resuming the proof of Theorem~\ref{theorem:RdmkFabmk}(d), we note
that~\eqref{eqn:JabA0CDprime},~\eqref{eqn:monominDdm}, and
Lemma~\ref{lemma:Aprimekge0}(b) imply that
\[
  \det \Jcal_{d,m,k}(\bfa,\bfb)
  = \det A_{d,m,k}'(\bfb) \det D_{d,m}(\bfb)
\]
has a monomial term of the form
\[
  m^{d+1}(-1)^{(m+1)d^2+dk} b_0^{m-1+k}(b_1b_2\cdots b_{d-1})^{2m-2} b_d^{2m-2-k}.
\]
In particular, we conclude that~$\gR_{d,m,k}'$ is a dominant rational
self-map of~$\PP^{2d}$ over~$\ZZ[m^{-1}]$. 
\par
The next step of the proof is to note that that the restriction
of~$\gR_{d,m,0}$ to~$\PP^{2d}=\{a_0=0\}\subset\PP^{2d+1}$ gives a
map~$\gR_{d,m,0}'$ from~$\{a_0=0\}$ to itself. This follows from the
fact that $G_{\bfa,\bfb,m,0}=\pm a_0b_0^{m-1}z^d+O(z^{d-1})$.  We
claim that~$\gR_{d,m,0}'$ is a dominant rational map over~$\ZZ$.  To
see this, we first perform a Jacobian calculation for~$\gR_{d,m,0}'$
similar to those already done to check that~$\gR_{d,m,0}'$ is a
dominant rational map over~$\ZZ[m^{-1}]$. We then
decompose~$\gR_{d,m,0}'$ as a composition of maps~$\gR_{d,p,0}'$
with~$p$ prime, cf.~\eqref{eqn:RdmprodRdpi}, so it suffices to
show that~$\gR_{d,p,0}'$ is dominant over~$\FF_p$. The proof of this
last assertion is similar to the final step in the proof of
Theorem~\ref{theorem:RdmkFabmk}(c) earlier in this section.
\par
Finally, writing~$m=np$ with~$p$ prime and using the
decomposition $\gR_{d,m,k}'=\gR_{d,n,0}'\circ\gR_{d,p,k}'$, we are
reduced to showing that~$\gR_{d,p,k}'$ is dominant over~$\FF_p$.  But
we know from Lemma~\ref{lemma:Aprimekge0}(b) that $\det
A_{d,m,k}'(\bfb)$ includes a monomial whose coefficient is~$\pm1$,
while the matrix~$D_{d,m}(\bfb)$ does not depend on~$k$, so the
argument used at the end of the proof of
Theorem~\ref{theorem:RdmkFabmk}(c) earlier in this section can be
used, \emph{mutatis mutandis}, to complete the proof of
Theorem~\ref{theorem:RdmkFabmk}(d).

\begin{remark}
Let $A_{d,m,k}'(\bfb)$ be the $d$-by-$d$ matrix described
in~\eqref{eqn:JabA0CDprime}, i.e., the matrix obtained by deleting
the first column and row of~$A_{d,m,k}(\bfb)$. For $k=0$, we have
\[
  \det A_{d,m,0}(\bfb) = (-b_0)^{m-1}\det A_{d,m,0}'(\bfb),
\]
since only the first entry of the first column of $A_{d,m,0}(\bfb)$
is nonzero. It appears from examples that the following formula is true:
\begin{equation}
  \label{eqn:detAdmkdetAdm0}
  b_d^k \det A_{d,m,k}'(\bfb) \stackrel{?}{=} 
    (-1)^{(d+1)k} b_0^k \det A_{d,m,0}'(\bfb).
\end{equation}

We also remark that an easy calculation shows that if $1\le m\le d+1$,
then the top $d$ rows of $A_{d,m,k+1}$ are equal to the bottom $d$
rows of $A_{d,m,k}$, while the last row of $A_{d,m,k+1}$ is equal to the
$(d+1-m)$'th row of~$A_{d,m,k}$ shifted one place to the right. But
this relation between~$A_{d,m,k}$ and~$A_{d,m,k+1}$ is not sufficient
to explain~\eqref{eqn:detAdmkdetAdm0}.
\end{remark}

\section{The map on $\PP^d$ induced by $H_{\bfb,m}$}
\label{section:mapinducedbyHbm}
In this section we discuss the map on~$\PP^d$ induced by using only
the denominator of the Landen transform. We write the
function~$H_{\bfb,m}(z)$ defined by the formula~\eqref{eqn:Hbmwm} in
Proposition~\ref{prop:resultantGH} as
\[
  H_{\bfb,m}(z) = \sum_{i=0}^d H_{m,i}(\bfb)z^{d-i},
\]
and we use the $z$-coefficients of~$H_{\bfb,m}(z)$ to define a map
\begin{equation*}
  h_m:\PP^d\longrightarrow\PP^d,\quad
  h_m(\bfb)=\bigl[H_{m,0}(\bfb),\ldots,H_{m,d}(\bfb)\bigr].
\end{equation*}
Thus~$h_m$ is the map formed using the final~$d+1$ coordinate
functions of the rational map~$\gR_{d,m,k}$ described in
Theorem~\ref{theorem:RdmkFabmk}(a). We start with an easy fact.

\begin{proposition}
The map~$h_m$ is a morphism.
\end{proposition}
\begin{proof} 
The computation~\eqref{eqn:ZRdmkcomp} done during the course of
proving Theorem~\ref{theorem:RdmkFabmk}(b) shows that
\[
  H_{m,0}(\bfb)=\cdots=H_{m,d}(\bfb)=0
  \;\Longleftrightarrow\;
  \bfb=\bfzero,
\]
so $\Zcal(h_m)=\emptyset$.
\end{proof}

More interesting is the fact that~$h_m$ is closely related to
the~$m$'th-power map. In order to describe the exact relationship, we
dehomogenize by setting $b_0=1$. Since
$H_{m,0}(\bfb)=(-1)^{(m+1)d}b_0^d$, this has the effect of
restricting~$h_m$ to an affine morphism (which by abuse of notation we
also call~$h_m$) given by
\begin{align}
  \label{eqn:hmAAd}
  h_m&:\AA^d\longrightarrow\AA^d, \notag\\*
  h_m(\bfb) 
     &= \bigl((-1)^{(m+1)d}H_{m,1}(\bfb),\ldots,(-1)^{(m+1)d}H_{m,d}(\bfb)\bigr),
\end{align}
where again by abuse of notation, we now use affine coordinates
$\bfb=(b_1,\ldots,b_d)$.  We let
\[
  \pi_m:\AA^d\longrightarrow\AA^d,\quad
  \pi_m(x_1,\ldots,x_d)=(x_1^m,\ldots,x_d^m)
\]
be the $m$-power map, and we let
\[
  \s_d:\AA^d\to\AA^d,\quad
  \s_d(\bfx)=\bigl(-\s_d^{1}(\bfx),\s_d^{2}(\bfx),
                     \ldots,(-1)^d\s_d^{d}(\bfx)\bigr),
\]
be the map defined by the elementary symmetric functions, taken with
alternating sign.  So for example, when $d=3$ we have
\[
  \s_3(x_1,x_2,x_3)
  = \bigl( -x_1-x_2-x_3, x_1x_2+x_1x_3+x_2x_3, -x_1x_2x_3\bigr).
\]
\par
It is well known that~$\s_d$ induces an isomorphism that takes the
quotient of~$\AA^d$ by the action of the symmetric group~$\Scal_d$
on~$(x_1,\ldots,x_d)$ to~$\AA^d$.  We denote this isomorphism by
\[
  \bar\s_d : \AA^d/\Scal_d \xrightarrow{\;\;\sim\;\;} \AA^d.
\]

\begin{proposition}
\label{proposition:hmbarspimbarsinv}
With notation as described in this section, we have
\[
  h_m = \bar\s_d\circ\pi_m\circ\bar\s_d^{-1}
  \quad\text{as self-maps of $\AA^d$.}
\]
\end{proposition}
\begin{proof} 
In order to relate~$h_m$ to~$\pi_m$, it is convenient to
let~$u_1,\ldots,u_d$ be the $z$-roots of~$F_\bfb(z)=0$ in some
integral closure of ${\ZZ[b_1,\ldots,b_d]}$. (Remember that we have
set $b_0=1$.) In other words,
\[
  b_i = (-1)^i \s_d^{i}(u_1,\ldots,u_d)
  \quad\text{and}\quad
  F_\bfb(z) = \prod_{s=1}^d (z-u_s).
\]
This allows us to compute
\begin{align*}
  \sum_{i=0}^d H_{m,i}(\bfb)w^{m(d-i)}
   &= H_{\bfb,m}(w^m) 
   =\prod_{t=0}^{m-1} F_\bfb(\z_m^t w)  
   = \prod_{t=0}^{m-1} \prod_{s=1}^d (\z_m^t w - u_s)  \\
   &=  \prod_{s=1}^d \prod_{t=0}^{m-1}  (\z_m^t w - u_s)  
   =  \prod_{s=1}^d  (-1)^{m+1} (w^m - u_s^m)  \\
   &=  (-1)^{d(m+1)}\sum_{i=0}^d (-1)^i  \s_d^{i}(u_1^m,\ldots,u_d^m) w^{m(d-i)}  \\
   &=  (-1)^{d(m+1)}\sum_{i=0}^d (-1)^i \s_d^{i}\circ\pi_m(u_1,\ldots,u_d) w^{m(d-i)}.
\end{align*}
Hence for $1\le i\le d$ we have
\[
  H_{m,i}(\bfb) = (-1)^{d(m+1)} (-1)^i \s_d^{i}\circ\pi_m(u_1,\ldots,u_d).
\]
Combining the~$H_{m,i}$ to form~$h_m$ as in~\eqref{eqn:hmAAd}, we
obtain
\[
  h_m(\bfb) = \bar\s_d\circ\pi_m(u_1,\ldots,u_d).
\]
However, by construction~$(u_1,\ldots,u_d)$ satisfies
\[
  \bar\s_d(u_1,\ldots,u_d) = \bfb,
\]
which gives the desired formula
\[
  h_m = \bar\s_d\circ\pi_m\circ\bar\s_d^{-1}.
\]
This completes the proof of
Proposition~\ref{proposition:hmbarspimbarsinv}.
\end{proof}

\begin{remark}
Although the equality in
Proposition~\ref{proposition:hmbarspimbarsinv} is only valid on the
affine set $\{b_0\ne0\}$, we note that on the excluded
hyperplane~$b_0=0$, the definition of~$h_m$ is given via a product of
polynomials of degree~$d-1$. This allows us to completely
describe~$h_m$ on~$\PP^d$ via a natural decomposition. 
More precisely, we decompose~$\PP^d$ as a disjoint union
\[
  \PP^d = \bigcup_{n=0 }^{d} \{ b_0=\cdots=b_{n-1}=0~\text{and}~b_n\ne0 \}
  = \bigcup_{n=0 }^{d} \AA^n.
\]
Then Proposition~\ref{proposition:hmbarspimbarsinv} applied to each
affine piece implies that the restriction of~$h_m:\PP^d\to\PP^d$
to~$\AA^n$ satisfies
\[
  \left.h_m\right|_{\AA^n} = \bar\s_n\circ\pi_m\circ\bar\s_n^{-1}.
\]
\end{remark}

\begin{remark}
In some situations it may be more convenient to view $\gR_{d,m,k}$
itself as a map of affine space by dehomogenizing with $b_0=1$. We
write $\gR_{d,m,k}'$ for the resulting polynomial map
\[
  \gR_{d,m,k}' : \AA_\ZZ^{2d+1} \longrightarrow \AA_\ZZ^{2d+1},
\]
where we identify
\[
  \AA_\ZZ^{2d+1} \xrightarrow{\;\sim\;}
   \PP_\ZZ^{2d+1} \setminus \{b_0=0\}.
\]
We observe that there are a number of vector subspaces
of~$\AA^{2d+1}_\ZZ$ that the morphism~$\gR_{d,m,k}'$ leaves invariant,
including for example
\begin{align*}
   U_i &= \{a_0=a_1=\cdots=a_i=0\} \quad\text{for $0\le i\le d$}, \\
   V_i &= \{a_d=a_{d-1}=\cdots=a_{d-i}=0\} \quad\text{for $0\le i\le d$}, \\
   W_i &= \{b_d=b_{d-1}=\cdots=b_{d-i}=0\} \quad\text{for $0\le i< d$}.
\end{align*}
(Note that for~$1\le k<m$, we have
$\gR_{d,m,k}'(\AA^{2d+1}_\ZZ)=U_0$.)  
\par
Writing~$\bfa$ as a column vector and letting~$A_{d,m,k}(\bfb)$ be the
matrix appearing in the Jacobian, see~\eqref{eqn:JabA0CD}, we see
that~$\gR_{d,m,k}'$ takes the form
\[
  \gR_{d,m,k}'(\bfa,\bfb) = \bigl(A_{d,m,k}(\bfb)\bfa,h_m(\bfb)\bigr).
\]
With this notation, the composition law for the Landon transform
described in Proposition~\ref{prop:compositionlaw} becomes the matrix
formula
\[
  A_{d,m,k}\bigl(h_m(\bfb)\bigr) A_{d,n,\ell}(\bfb)
  = A_{d,mn,kn+\ell}(\bfb).
\]
\end{remark}




\end{document}